\documentclass[a4paper]{article}
\usepackage{amsmath}
\usepackage{amsthm}
\usepackage{amssymb}
\usepackage{graphicx}
\usepackage{authblk}
\usepackage{hyperref}
\usepackage{multirow}
\usepackage{subfigure}
\usepackage{cite}
\usepackage{color} 
\usepackage[dvipsnames]{xcolor}
\usepackage{comment}
\usepackage{enumitem}
\newtheorem{theorem}{Theorem}[section]
\newtheorem{proposition}[theorem]{Proposition}
\newtheorem{assertion}[theorem]{Assertion}
\newtheorem{lemma}[theorem]{Lemma}

\newtheorem{remark}{Remark}

\usepackage{fullpage}
\newcommand{\eqlab}[1]{\label{eq:#1}}
\renewcommand{\eqref}[1]{(\ref{eq:#1})}

\newcommand{\Wert}{{\vert\kern-0.25ex\vert\kern-0.25ex\vert}}

\newcommand{\figref}[1]{Fig.~\ref{fig:#1}}
\newcommand{\figlab}[1]{\label{fig:#1}}
\newcommand{\propref}[1]{Proposition~\ref{proposition:#1}}
\newcommand{\proplab}[1]{\label{proposition:#1}}
\newcommand{\lemmaref}[1]{Lemma~\ref{lemma:#1}}
\newcommand{\lemmalab}[1]{\label{lemma:#1}}
\newcommand{\remref}[1]{Remark~\ref{remark:#1}}
\newcommand{\remlab}[1]{\label{remark:#1}}
\newcommand{\thmref}[1]{Theorem~\ref{theorem:#1}}
\newcommand{\thmlab}[1]{\label{theorem:#1}}
\newcommand{\assref}[1]{Assertion~\ref{ass:#1}}
\newcommand{\asslab}[1]{\label{ass:#1}}

\newcommand{\appref}[1]{Appendix~\ref{app:#1}}

\newcommand{\applab}[1]{\label{app:#1}}

\newcommand{\secref}[1]{section~\ref{sec:#1}}
\newcommand{\seclab}[1]{\label{sec:#1}}

\newcommand\rsp[1]{{\color{black}{#1}}}
\newcommand\rspp[1]{{\color{black}{#1}}}
\author{Kristian Uldall Kristiansen}

\affil{Department of Applied Mathematics and Computer Science, Technical University of Denmark, 2800 Kgs. Lyngby, Denmark}
\affil{{\tt krkri@dtu.dk}}

\title{Improved Gevrey-1 estimates of formal series expansions of center manifolds}

\begin{document}

\maketitle
\date{}

\begin{abstract}
In this paper, we show that the coefficients $\phi_n$ of the formal series expansions $\sum_{n=1}^\infty \phi_n x^n\in x\mathbb C[[x]]$ of center manifolds of planar analytic saddle-nodes grow like $\Gamma(n+a)$ (after rescaling  $x$) as $n\rightarrow \infty$. Here the quantity $a$ is the formal analytic invariant associated with the saddle-node (following the work of J. Martinet and J.-P. Ramis). This growth property of $\phi_n$, which \rspp{is optimal}, was recently (2024) described for a restricted class of nonlinearities by the present author in collaboration with P. Szmolyan. This joint work was in turn inspired by the work of Merle, Rapha\"{e}l, Rodnianski, and
               Szeftel (2022), which described the growth of the coefficients for a system related to self-similar solutions of the compressible Euler. In the present paper, we combine the previous approaches with a Borel-Laplace approach. Specifically, we adapt the Banach norm of Bonckaert and De Maesschalck (2008) in order to capture the singularity in the complex plane. \rsp{Finally, we apply the result to a family of Riccati equations and obtain a partial classification of the analytic center manifolds.}
              %
%
%
\end{abstract}


\tableofcontents
\section{Introduction}
In this paper, we consider analytic saddle-nodes of the form
\begin{equation}
\begin{aligned}
 \dot x &= x^2,\\ \dot y &=-(1+ax) y + f(x,y),
 \eqlab{system00}
\end{aligned}
\end{equation}
with $f(0,0)=f'_y(0,0)=f''_{xy}(0,0)=0$. This form can easily be obtained from any saddle-node (with Poincar\'e rank equal $1$) by elementary transformations (including transformations of time), see \secref{nf}. It is well-known that the center manifolds, as graphs over $x$: $y=\phi(x)$, $x\in (-\delta,\delta)$, $\delta>0$, are non-analytic in general, see e.g. \cite{dumortier2006a}. Instead, their formal series expansion 
\begin{align*}
 \phi = \sum_{n=1}^\infty \phi_n x^n\in x \mathbb C[[x]],
\end{align*}
are Gevrey-1, i.e. there are $A,T>0$ such that 
\begin{align}
 \vert \phi_n\vert\le A T^{n-1} (n-1)! \quad \forall\, n\in \mathbb N.\eqlab{gevrey1}
\end{align}
For example, for the system
\begin{align}
 x^2 \frac{dy}{dx} = -y+x,\eqlab{euler}
\end{align}
often attributed to Euler, 
a simple calculation shows that $\phi_n = (-1)^{n-1} (n-1)!$, which illustrates the lack of analyticity and the validity of the estimate \eqref{gevrey1}.
In fact, a stronger result holds in general: The Borel transform $\Phi=\mathcal B(\phi)$ of $\phi$:
\begin{align*}
 \Phi(w) = \sum_{n=0}^\infty \frac{\phi_n}{(n-1)!}w^{n-1},
\end{align*}
which is absolutely convergent for $\vert w\vert<T^{-1}$ by \eqref{gevrey1}, can be endlessly analytically continued, along any ray $re^{i\theta}$, $r\ge 0$, in $\mathbb C$ with direction $\theta\ne -\pi$ so that $\vert \Phi(w)\vert\le C e^{P\vert w\vert}$, $C,P>0$.
This means that the series is $1$-summable, see e.g. \cite{balser1994a,bonckaert2008a}, and that $\phi$ can be realized, \rspp{as an analytic function} on a local sectorial domain $\omega$, by the Laplace transform $\mathcal L$ of (the analytically extended) $\Phi(w)$:
\begin{align*}
 \phi(x) = \int_0^{\infty} e^{-x^{-1}w} \Phi(w) dw,\quad x\in \omega.
\end{align*}
We refer to further details below and to \cite{bonckaert2008a}. For \eqref{euler}, we have 
$\Phi(w) = \frac{1}{1+w}$ and $\phi(x) = \int_0^\infty e^{-x^{-1} w} \frac{1}{1+w}dw$.

\subsection{Main result}
In this paper, we prove the following:
\begin{theorem}\thmlab{main}
 Let $\phi=\sum_{n=1}^\infty \phi_n x^n\in x\mathbb C[[x]]$ denote the unique formal series expansion of the center manifold of \eqref{system00}. Then there exists a number $S_\infty\in \mathbb R$ so that 
 \begin{align}
  \frac{(-1)^n \phi_n}{\Gamma(n+a)}\rightarrow S_\infty\quad \mbox{for}\quad  n\rightarrow \infty.\eqlab{mnasymp}
 \end{align}
\end{theorem}

To illustrate how \eqref{mnasymp} occurs, we first consider the case $f(x,y)=f(x)$ (where $f$ is independent of $y$) in \eqref{system00}:
\begin{align}
 x^2 \frac{dy}{dx}+(1+ax)y=f(x),\quad f(x)=\sum_{n=1}^\infty f_n x^n.\eqlab{findpy}
\end{align}
This differential equation for $y=\phi(x)$ is linear in $y$ and one can solve explicitly for the $\phi_n$'s of the formal series. Indeed, inserting the formal series $y=\sum_{n=1}^\infty \phi_n x^n\in x\mathbb C[[x]]$  into \eqref{findpy}
leads to 
\begin{align*}
 \sum_{n=1}^\infty \left( (n +a)\phi_{n} x^{n+1}+ \phi_n x^n\right)=\sum_{n=1}^\infty f_n x^n,
\end{align*}
and therefore to 
the recursion relation:
\begin{align}
 \phi_n + (n-1+a)\phi_{n-1} = f_n\quad \forall \,n\in \mathbb N,\eqlab{recmk0}
\end{align}
with $\phi_0= 0$. 
\begin{lemma}\lemmalab{mk0linearcase}
Let $\Gamma:\mathbb C\backslash (-\mathbb N_0)\rightarrow \mathbb C$ denote the gamma function, see \appref{gamma}. Moreover, suppose that $a>-1$ and define 
\begin{align}
 S_n :=\sum_{j=1}^n \frac{(-1)^j f_j}{\Gamma(j+a)}.
\end{align}
Then the solution of the recursion relation \eqref{recmk0} with $\phi_0=0$ is
 \begin{align}
 \phi_n  = (-1)^n \Gamma(n+a) S_n.
\end{align}
\end{lemma}
\begin{proof}
 The result can easily be proven by induction using the base case $\phi_0=0$ and the basic property of $\Gamma$: $\Gamma(z+1)=z \Gamma(z)$, in the induction step. 
%
%
\end{proof}
Since $f$ is analytic, we have 
\begin{align*}
 \vert f_k\vert \le A T^{k},
\end{align*}
for some $A,T>0$,
and the sum
\begin{align}
 S_\infty := \lim_{n\rightarrow \infty} S_n=\sum_{n=1}^\infty \frac{(-1)^j f_j}{\Gamma(j+a)},\eqlab{Sinfty0}
\end{align}
is therefore absolutely convergent for any $a>-1$:
\begin{align*}
 \vert S_\infty\vert \le \sum_{j=1}^\infty \frac{\vert f_j\vert}{\Gamma(j+a)}\le A\sum_{j=1}^\infty \frac{T^{j}}{\Gamma(j+a)}<\infty.
\end{align*}
Here we have used Stirling's formula, see \eqref{stirling0} below, for the boundedness of the majorant series.
This calculation proves \eqref{mnasymp} in the case where $f$ is independent of $y$ and $a>-1$. 
For $a\le -1$, we first perform a translation: $y=\sum_{n=1}^{M} \phi_n x^n+\tilde y$ with $M=\lceil -a\rceil$ so that 
\begin{align}
 x^2 \frac{d\tilde y}{dx}+(1+ax)\tilde y=\tilde f(x),\quad \tilde f(x):=- (M+a)\phi_{M}x^{M+1}+ \sum_{k=M +1}^\infty f_k x^k.
\end{align}
Then by proceeding as above for $a>-1$, we can complete the proof of \eqref{mnasymp} in the case where $f$ is independent of $y$. The remainder of the paper is devoted to the fully nonlinear case.

\subsection{\rsp{Relation to previous work}}
In \cite{MR4445442}, the authors construct certain $C^\infty$-smooth self-similar solutions of the isentropic ideal compressible Euler equations. These solutions  were used in \cite{merle2022a} to determine finite time energy blowup solutions of Navier-Stokes equations (isentropic ideal compressible), see also \cite{merle2022b} for applications to the defocusing nonlinear Schr{\"o}dinger equation. In essence, see also \cite{euler}, the $C^\infty$-smooth solutions were constructed from a bifurcation problem involving an analytic weak-stable invariant manifold of a nonresonant node. These invariant manifolds are delicate objects but the authors could control these near saddle-node bifurcations (where the resonances of the node accumulate). This description rested upon a statement like \eqref{mnasymp} for the center manifold at the saddle-node. In particular, it turns out \rsp{{that the sign of $S_\infty \ne 0$ (so that the center manifold is nonanalytic) relates to the position of the analytic weak-stable invariant manifold for parameter values near the saddle-node}} (see also \cite[\rsp{Corollary 3.6}]{euler}). The authors of \cite{MR4445442} proved their version of \eqref{mnasymp} for their specific (\rspp{rational}) nonlinearity using \rspp{the implicit equation (in a fix-point formulation)} for the coefficients $\phi_n$ that arises from solving \eqref{recmk0} for $\phi_n$, see \lemmaref{mk0linearcase}. Subsequently, in \cite{euler} the authors provided a dynamical systems oriented approach to the phenomena of \cite{MR4445442} for a general set of equations. However, the authors imposed a \rsp{significant restriction} on the nonlinearity in order to obtain \eqref{mnasymp}. In particular, they were not able to generalize the approach of \cite{MR4445442}. \rsp{In further details, the paper \cite{euler} writes $f(x,y)$ in the form $f(x)+\mu h(x,y)$ with $f(x)=\mathcal O(x^2)$, $h(x,y)=\mathcal O(x^2 y,y^2)$ and assumes that $0<\mu\ll 1$. In this way, \eqref{mnasymp} is essentially obtained by perturbing away from the $\mu=0$ case, which is described in \lemmaref{mk0linearcase}. (The reference \cite{euler} also assume that $a>-2$ but they argue that this is not essential.)}

The quantity $S_\infty$ in \thmref{main} is somehow related (the connection is not clear to the author) to \textit{the translational part of the analytic invariants} of the saddle-node, see \cite{martinet1983a}. Indeed, the center manifold is analytic if and only if this translational part is trivial, see also \cite[Section 1]{loray2004a}. However, whereas $S_\infty\ne 0$ clearly implies that the center manifold is nonanalytic, it is at this stage not known (to the best of my knowledge) whether the converse is true in general (it holds in the $y$-linear case, \rsp{see \cite[Lemma 2.4]{euler}}). \rsp{We state this as a general research question}:

\

\rsp{\textbf{Question I}: Is $\phi \in x\mathbb C[[x]]$ convergent if $S_\infty=0\Longrightarrow$?}

\

\rsp{It is a basic fact that $\phi\in x \mathbb C[[x]]$ is convergent if and only if the associated Borel transform $\Phi=\mathcal B(\phi)$ is an entire function with exponential growth, i.e. there is an $\epsilon>0$ small enough such that $\vert \Phi(w)\vert e^{-\epsilon^{-1} \vert w\vert}<\infty$ for all $w\in \mathbb C$, see \secref{BL}. Our proof of \thmref{main} shows that $\Phi$ has a singularity in the Borel plane if $S_\infty\ne 0$. But we do not know if $S_\infty$ captures the singularities completely so that if $S_\infty=0$ then there are no singularities of $\Phi$. We leave \textbf{Question 1} in full generality to future work, but do discuss it further below in the specific context of a family of Riccati equations (see \secref{riccati}). } 

%
%
%
%

\rsp{The use of Borel-Laplace for differential equations has a long history, dating back to \'Emile Borel in the nineteenth century, see \cite{wasow1965a} and the preface in \cite{mitschi2016a}. Later during the 1980s-1990s the theories of $k$-summability and multisummability were developed by  several authors J.-P. Ramis, Y. Sibuya, Martinet, Malgrange, Balser, Ecalle, etc., see \cite{balser1994a} and references therein,
which enabled a complete framework for asymptotic series of linear differential equations with analytic coefficients. Moreover, summability is central to the development of exact WKB, see \cite{voros1983a}. 
 Subsequently, Braaksma \cite{braaksma1992a} applied the theory of multisummability to nonlinear equations. Summability is also central (through accelero-summation) to Ecalle's proof of Dulac's theorem for polynomial differential equations, see \cite{ecalle}. 
More recently, Borel-Laplace has been used to study normal forms \cite{bonckaert2008a,new} as well as slow manifolds \cite{de2020a}.
In general terms, the Borel-Laplace method assigns analytic functions, that are defined on open and local sectorial domains, to each $k$-summable formal series. In general, the functions do not coincide on overlapping domains and this is related to Stokes phenomena, resurgence and more broadly to the presence of singularities in the complex Borel plane, see \cite{mitschi2016a}. 

The constant $S_\infty$ is reminiscent of \textit{Stokes constants} that are frequently associated with exponentially small splitting phenomena. For example, in the case of the unfolding of the zero-Hopf bifurcation, see e.g. \cite{baldom2013a}, the situation is similar to \cite{euler,merle2022a} (and analytic weak-stable  manifolds near saddle-nodes): The Stokes constant is obtained from the unperturbed system $\epsilon=0$, and when it is nonzero, then the leading order exponentially small splitting of the invariant manifolds can be determined. As discussed in \cite[Section 4]{new}, a nonzero Stokes constant also relates to the lack of analyticity of center-like invariant manifolds of generalized saddle-nodes. 
 In all cases, that I am aware of, distinguished solutions of ``inner equations'' (see e.g. \cite{baldoma2012a}), that are frequently associated to such problems, can be obtained from Borel-Laplace theory. For example, \cite[Theorem 4]{baldom2013a} follows from \cite[Theorem 3]{bonckaert2008a} (as a corollary of the invariance of $y=0$ for the normal form \cite[Eq. (8)]{bonckaert2008a}). The same is true for difference equations, see \cite{baldoma2012a} and \cite{gelfreich2001a}. 

Parallel to these rigorous studies, there is the more applied method of exponential asymptotics, \cite{chapman1998a}. It is my understanding that this approach primarily relates to optimal truncation of the divergent series, using ``factorial-over-power ansatz'' for the late terms of series, see \cite{dingle1973a}. (Optimal truncation was also crucial in \cite{euler} for $\epsilon>0$ due to the resonances of the node.) Notice that by Stirling's formula, see e.g. \eqref{stirling} below, $\phi_n$ in \thmref{main} also has factorial-over-power growth as $n\rightarrow \infty$. Exponential asymptotics has proven very successful in recent years as a practical tool. It has also been used to describe phenomena (like snaking in \cite{chapman2009a}) that seem to escape more rigorous approaches.
Exponential asymptotics is also intrinsically related to resurgence and Borel-Laplace, see \cite{crew2024a}.

} 

\subsection{Organization of the paper}
The paper is organized as follows: In \secref{nf}, we present a convenient normal form. In this normal form, it will be important that $a>1$ and that $\phi = \sum_{k=2}^\infty \phi_n x^n\in x^2\mathbb C[[x]]$ (see \remref{resonance} below). \rsp{In \secref{BL}, we then review the Borel-Laplace approach of \cite{bonckaert2008a} and in \secref{adapt} we present an adaptation of their Banach space norm that enables us to describe the details of the singularity of the Borel transform $\Phi$ of $\phi$ in the Borel plane. This leads to an initial estimate of $\phi_n$ for $n\rightarrow \infty$, see \lemmaref{borellaplaceest}, which we then bootstrap in  \secref{boot} to complete the proof of \thmref{main}}.  \rsp{In \secref{riccati}, we then combine the theory with numerical computations in order to obtain a partial characterization of analytic center manifolds of a family of Riccati equations:}
\begin{align}
 \rsp{x^2 \frac{dy}{dx} +(1+a x) y = b x^2+ c y^2},\eqlab{riccati0}
\end{align}
\rsp{with parameters $a,b,c\in \mathbb R$}.
Finally, in \secref{discuss} we discuss the result of the paper. \rspp{Important properties of the gamma function $\Gamma:\mathbb C\setminus (-\mathbb N_0)\rightarrow \mathbb C$ are collected in \appref{gamma}}.

\subsection{Acknowledgement}
The author wishes to thank Peter De Maesschalck for useful discussions on the Borel-Laplace approach and for feedback on an earlier version. Moreover, the author wishes to thank Peter Szmolyan for the fruitful collaboration on \cite{euler}, which laid the foundation for the present work. 
\section{A normal form}\seclab{nf}

Consider any planar real-analytic system having a saddle-node at the origin. Let $\lambda\in \mathbb R\backslash \{0\}$ denote the single nonzero eigenvalue of the linearization. 
Then there exists local coordinates such that 
\begin{equation}\eqlab{xysn}
\begin{aligned}
\dot x &= g(x,y),\\
\dot y &= \lambda y+h(x,y),
\end{aligned}
\end{equation}
where neither $g$ nor $h$ contain constant or linear terms. Suppose that 
\begin{align*}
 g''_{xx}(0,0) \ne 0.
\end{align*}
Then the blow-up transformation $(x,\tilde y)\mapsto y$ defined by $y=x \tilde y$ brings \eqref{xysn} into the following locally defined system:
\begin{align*}
  \frac12 g''_{xx}(0,0) x^2 \frac{d\tilde y}{dx} &=  \lambda \tilde y+\mathcal O(x,x\tilde y,\tilde y^2).
\end{align*}
(Notice that the transformation is of blow-up type since $(0,\tilde y)\mapsto y=0$.)
Since we are interested in center manifolds $y=\phi(x)$, we have \rspp{here} eliminated time.
The remainder terms can clearly be expressed in terms of $g$ and $h$, but this will not be relevant. Now, we finally put $$x=\frac{-2\lambda}{g''_{xx}(0,0)}\tilde x,$$ and obtain \eqref{system00} in the form
\begin{equation}\eqlab{system01}
\begin{aligned}
  x^2  \frac{dy}{dx}&=- (1+a x)  y + f( x, y),
\end{aligned}
\end{equation}
with $a\in \mathbb R$, $f$ real-analytic and $f(0,0)= f'_y(0,0)= f''_{xy}(0,0)=0$ after dropping the tildes. 

We now proceed to normalize \eqref{system01} further.
\begin{lemma}\lemmalab{this1}
 Consider \eqref{system01} with $f$ real-analytic and let $y=\sum_{n=1}^\infty \phi_n x^n\in x\mathbb C[[x]]$ denote the formal series expansion of the center manifold. Define $M=\max (2,2+\lceil -a\rceil)$ so that $a_M:=a+M\ge 2$. Then there is a locally defined real-analytic function $q=q(x)$, $q(0)=1$, such that the blow-up transformation $(x,\tilde y)\mapsto y$ defined by
 \begin{align}
  y =  \sum_{n=1}^{M+1} \phi_n x^n+x^{M} q(x) \tilde y,\eqlab{final}
 \end{align}
brings \eqref{system01} into the form
\begin{align}
 x^2 \frac{d\tilde y}{dx} = -(1+a_Mx)\tilde y + x^2 f_0(x)+x^2 \tilde y^2 f_2(x,\tilde y),\eqlab{system0}
\end{align}
where $f_0$ and $f_2$ are locally defined real-analytic functions.
\end{lemma}
\begin{proof}
Consider first the translation of $y$, given by
  $y=\sum_{n=1}^{N+1} \phi_n x^n+\tilde y$, for any $N\in \mathbb N$. Then by construction
\begin{align}
 x^2 y' =- (1+ax)y+ x^{N+2} f_0(x)+ yf_1(x)  + y^2f_2(x,y),\eqlab{yeqn02}
\end{align}
after dropping the tildes, where $f_1(0)=f_1'(0)=0$. 

Next, define $q(x) = e^{\int_0^x s^{-2} f_1(s) ds}$, which is well-defined and analytic since $f_1(x)=\mathcal O(x^2)$. In particular, $q(0)=1$. Then it is a simple calculation to shows that $y=q(x)\tilde y$ leads to the form \eqref{yeqn02}, upon removing the tildes, with $f_1=0$, i.e.
\begin{align}
 x^2 y' =- (1+ax)y+ x^{N+2} f_0(x) + y^2f_2(x,y),\eqlab{yeqn03}
\end{align}
for some new locally defined analytic functions $f_0$ and $f_2$. 
Finally, we introduce the blow-up $y=x^N \tilde y$. Then \eqref{yeqn03} becomes
\begin{align*}
 x^2 y' = -(1+(a+N) x)y + x^{2} f_0(x) + x^N y^2 f_2(x,x^N y),
\end{align*}
upon dropping the tildes. The result then follows upon taking $N=M:=\max(2, 2+\lceil - a\rceil )$.
\end{proof}
%
%

\begin{lemma}\lemmalab{this2}
 Let $M=\max (2,2+\lceil -a \rceil)$, $a_M=a+M\ge 2$, $\phi_1,\ldots,\phi_{M+1}\in \mathbb R$, $q:B_r(0)\rightarrow \mathbb C$, $r>0$, analytic and $q(0)=1$, and suppose that $\tilde y=\sum_{n=2}^\infty \widetilde \phi_n x^n\in x^2\mathbb C[[x]]$ satisfies
 \begin{align*}
  \frac{(-1)^n \widetilde \phi_n}{\Gamma(n+a_M)}\rightarrow S_\infty\quad \mbox{for}\quad n\rightarrow \infty,
 \end{align*}
for some $ S_\infty\in \mathbb R$. 
Then $y=\sum_{n=1}^\infty \widehat \phi_n x^n\in x\mathbb C[[x]]$ defined by \eqref{final}:
\begin{align*}
 \sum_{n=1}^\infty \widehat \phi_n x^n:=\sum_{n=1}^{M+1} \phi_n x^n+x^{M} q(x) \sum_{n=2}^\infty \widetilde \phi_n x^n,
\end{align*}
satisfies the statement of \thmref{main}, i.e.
\begin{align*}
  \frac{(-1)^n \widehat \phi_n}{\Gamma(n+a)}\rightarrow S_\infty\quad \mbox{for}\quad n\rightarrow \infty.
 \end{align*}
\end{lemma}
\begin{proof}
We first consider the formal series defined by the product
\begin{align*}
 \sum_{n=2}^\infty  \widetilde \phi^q_n x^n:=q(x) \sum_{n=2}^\infty \widetilde \phi_n x^n\in x^2 \mathbb C[[x]],
\end{align*}
and write 
\begin{align*}
 q(x) = \sum_{n=0}^\infty q_n x^n,
\end{align*}
with $$\vert q_n\vert \le T^{n}\quad \mbox{for all}\quad n\in \mathbb N_0.$$ Here $q_0=1$ since by assumption $q(0)=1$. By Cauchy's product formula, we therefore obtain that
\begin{align}
 \widetilde \phi^q_n = \widetilde \phi_n+\sum_{k=2}^{n-1} \widetilde \phi_k q_{n-k}.\eqlab{mq}
\end{align}
We now verify that 
\begin{align}
 \frac{(-1)^n \widetilde \phi^q_n }{\Gamma(n+a_M)}\rightarrow S_\infty\quad \mbox{for}\quad n\rightarrow \infty.\eqlab{step1}
\end{align}
For this purpose, we write $\widetilde \phi_n = \Gamma(n+a_M) p_n$ with $\vert p_n\vert \le C$ for all $n\ge 2$ for some $C>0$. It follows from \eqref{mq} that 
\begin{align*}
 \frac{(-1)^n \widetilde \phi^q_n }{\Gamma(n+a_M)} - \frac{(-1)^n \widetilde \phi_n }{\Gamma(n+a_M)} = (-1)^n \sum_{k=2}^{n-1} \frac{\Gamma(k+a_M)p_k}{\Gamma(n+a_M)}q_{n-k},
\end{align*}
and to obtain \eqref{step1}, we will in the following show that the right hand side goes to zero as $n\rightarrow \infty$. 

Fix $n_0\gg 1$. It is clear that 
\begin{align*}
 \sum_{k=2}^{n_0-1} \frac{\Gamma(k+a)p_k}{\Gamma(n+a_M )}q_{n-k}\rightarrow 0\quad \mbox{for}\quad n\rightarrow \infty.
\end{align*}
We therefore consider 
\begin{align*}
\sum_{k=n_0}^{n-1} \frac{\Gamma(k+a_M )p_k}{\Gamma(n+a_M)}q_{n-k}.
\end{align*}
and estimate 
\begin{align*}
\left|\sum_{k=n_0}^{n-1} \frac{\Gamma(k+a_M)p_k}{\Gamma(n+a_M)}q_{n-k}\right|\le C\sum_{k=n_0}^{n-1} \frac{\Gamma(k+a_M)}{\Gamma(n+a_M)}T^{n-k}.
\end{align*}
Given that $n_0\gg 1$, we can use \eqref{stirling0}. Notice in particular that 
\begin{align*}
 \Gamma(k+a_M) \le (1+o_{n_0\rightarrow \infty}(1)) \sqrt{2\pi(n+a_M-1)} \left(\frac{n+a_M-1}{e}\right)^{k+a_M-1}\quad \forall\, n_0\le k\le n.
\end{align*}
This leads to 
\begin{align*}
\left|\sum_{k=n_0}^{n-1} \frac{\Gamma(k+a_M)p_k}{\Gamma(n+a_M)}q_{n-k}\right| &\le (1+o_{n_0\rightarrow \infty}(1)) C \sum_{k=n_0}^{n-1} \left(\frac{eT}{n+a_M-1}\right)^{n-k}\\
&\le 2C \left(\frac{eT}{n+a_M-1}\right) \sum_{k=0}^\infty \left(\frac{eT}{n+a_M-1}\right)^k\\
& \le 4C \left(\frac{eT}{n+a_M-1}\right),
\end{align*}
for all $n\ge n_0+1$, provided that $n_0\gg 1$ is such that $o_{n_0\rightarrow \infty}(1)<1$ and $0<\frac{eT}{n_0+a_M-1}\le \frac12$. We conclude that 
\begin{align*}
 \sum_{k=2}^{n-1} \frac{\Gamma(k+a_M)p_k}{\Gamma(n+a_M)}q_{n-k} \rightarrow 0\quad \mbox{for}\quad n\rightarrow \infty,
\end{align*}
which shows \eqref{step1}. 

Finally, we notice that
\begin{align*}
 x^M \sum_{n=2}^\infty \widetilde \phi_n^q x^n = \sum_{n=M+2}^\infty \widetilde \phi^q_{n-M}x^n
\end{align*}
and therefore
\begin{align*}
 \frac{(-1)^n \widetilde \phi^q_{n-M} }{\Gamma(n+a)} = \frac{(-1)^n \widetilde \phi^q_{n-M}}{\Gamma(n-M+a_M)}\rightarrow S_\infty\quad \mbox{for}\quad n\rightarrow \infty
\end{align*}
This completes the proof, since the difference between $\sum_{n=1}^\infty\widehat \phi_n x^n$ and $x^M \sum_{n=2}^\infty \widetilde \phi_n^q x^n $ is a finite sum, $\sum_{n=1}^{M+1}\phi_n x^n$. 

\end{proof}

It follows that in order to prove \thmref{main} it suffices to consider \eqref{system0} with $a_M\ge 2$. For simplicity, we now drop the subscript on $a_M$ and the tilde on $y$ and therefore consider
\begin{align}
 x^2 \frac{dy}{dx} +(1+a\rsp{x}) y =x^2 f_0(x)+x^2 y^2 f_2(x,y),\quad a\ge 2.\eqlab{System0}
\end{align}
Moreover, we let $B_r(z)\subset \mathbb C$ denote the open disc of radius $r>0$ centered at $z$.
\begin{remark}\remlab{ref1}
 \rsp{The reference \cite{euler} assumes that $a>-2$. \lemmaref{this1} and \lemmaref{this2} shows that this is without loss of generality in the context of \thmref{main}. However, the main focus of \cite{euler} is on the unfolding of the saddle-node  and here the details of $a\le -2$ has not been worked out yet. }
\end{remark}

\section{A Borel-Laplace approach to center manifolds}\seclab{BL}

%

In this section, we \rsp{review} the Borel-Laplace approach of \cite{bonckaert2008a} in the context of center manifolds. 
For this purpose, we first recall that the Borel transform $\mathcal B$ is defined in the following way: If $\phi=\sum_{n=1}^\infty \phi_n x^n\in x\mathbb C[[x]]$ is a Gevrey-1 formal series: 
\begin{align}
\vert \phi_n\vert \le A T^n (n-1)!\quad \forall\, n\in \mathbb \mathbb N,\eqlab{gevreyhn}
\end{align}
recall \eqref{gevrey1},
then 
\begin{align*}
 \mathcal B(\phi)(w) = \sum_{n=1}^\infty \frac{\phi_{n}}{(n-1)!} w^{n-1}\in \mathbb C[[w]].
\end{align*}
By \eqref{gevreyhn}, $\Phi = \mathcal B(\phi)$ is analytic on $B_{T^{-1}}(0)$. It is elementary to show that $\phi$ is convergent if and only if  $\Phi$ is an entire function with at most exponential growth of order $1$:
\begin{align*}
 \vert \phi_n \vert \le A T^{n-1} \Leftrightarrow \vert \Phi(w)\vert \le A e^{T\vert w\vert}\quad \forall\,w\in \mathbb C,
\end{align*}
see e.g. \lemmaref{Fest} below for a similar statement.

For the Laplace transform, on the other hand, we need analytic functions that are at most exponentially growing (of order $1$) on a sector. 
For simplicity, we will restrict attention to sectors $S_\alpha\subset \mathbb C$ centered along the positive real axis:
\begin{align*}
 S_\alpha =\left\{w=re^{i\theta}\in \mathbb C\,:\, 0\le \theta< \frac{\alpha}{2}\right\},
\end{align*}
and define
\begin{align}
 \Omega(\alpha,R) := S_{\alpha}\cup B_R(0),\quad 0<R<T^{-1}.\eqlab{Omega}
\end{align}
Here $\alpha\in (0,\pi)$ denotes the opening of the sector, see \figref{omega}. (Extensions to general sectors centered along other directions, that do not include the negative real axis, are left out but straightforward.) Suppose then that $\Phi=\mathcal B(\phi)$ can be analytically extended to $\Omega$, so that $\Phi:\Omega \rightarrow \mathbb C$ is analytic, and suppose that $\vert \Phi(w)\vert \le C e^{P\vert w\vert}$, $C,P>0$, for all $w\in \Omega$.
Then the Laplace transform 
\begin{align*}
 \mathcal L_{\theta}(\Phi)(x) = \int_0^{e^{i\theta} \infty} e^{-x^{-1} w} \Phi(w)dw,
\end{align*}
of $\Phi$ in the direction $\theta\in (-\alpha,\alpha)$, is well-defined
for $\vert x\vert<P^{-1}\cos (\theta-\arg(x))$. In particular, by analytic continuation, $\mathcal L_{\theta}(\Phi)(x)$ defines an analytic function -- which we denote by $\mathcal L(\Phi)$ -- on the local sector 
\begin{align*}
 \omega (\alpha,r):= S_{\alpha+\pi}\cap B_r(0),
\end{align*}
where $r=r(\alpha)\in (0,P^{-1})$ is small enough, see \cite[Proposition 3]{bonckaert2008a}.
 $\mathcal L(\Phi):\omega\rightarrow \mathbb C$ is called the $1$-sum of the Gevrey-1 series $\phi\in x\mathbb C[[x]]$, \cite{balser1994a}. 
In this regard, it is important to notice that
\begin{align*}
 \mathcal L((\cdot)^{n-1})(x) = (n-1)! x^n\quad \forall\, n\in \mathbb N,
\end{align*}
and therefore, at the level of formal series, $\mathcal L:\mathbb C[[w]]\rightarrow x\mathbb C[[x]]$ is the inverse of the Borel transform $\mathcal B:x\mathbb C[[x]]\rightarrow \mathbb C[[w]]$. 

\begin{figure}[h!]
\begin{center}
\subfigure[]{\includegraphics[width=.4\textwidth]{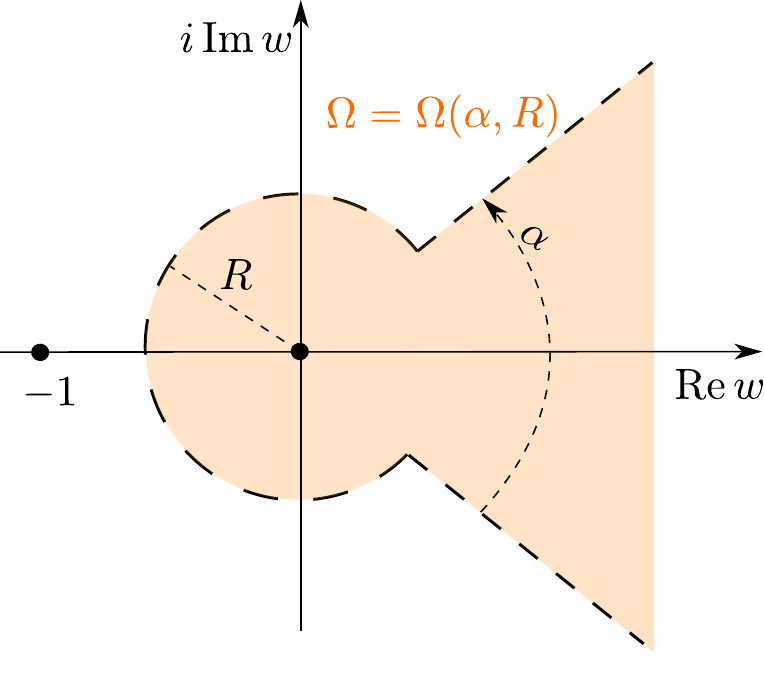}}
\subfigure[]{\includegraphics[width=.4\textwidth]{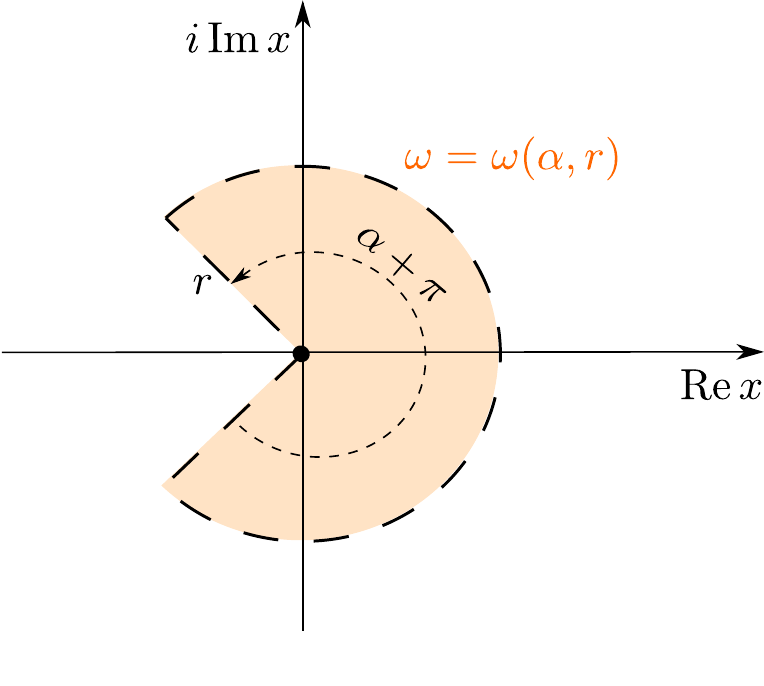}}
\end{center}
\caption{Illustrations of the different domains $\Omega(\alpha,R)$ and $\omega(\alpha,r)$. Fig. (a) is in the ``Borel plane'' $w\in \mathbb C$, whereas Fig. (b) is in the $x$-domain.}
\figlab{omega}
\end{figure}

The authors of \cite{bonckaert2008a} apply the Borel-Laplace approach to study (among other things) $1$-summable normal forms of saddle-nodes in $\mathbb C^{n+1}$ (with Poincar\'e rank $1$) and the existence of $1$-summable center manifolds (of complex dimension $1$) can be obtained as a corollary of these results (since $y=0$ is invariant for \cite[Eq. (8)]{bonckaert2008a}). In further details, the Borel-Laplace approach of \cite{bonckaert2008a} is centered around the following $\epsilon$-dependent norm
\begin{align}
\Vert Y\Vert_\epsilon:=\sup_{w\in \Omega} \left\{\vert Y(w)\vert e^{-\epsilon^{-1} \vert w\vert} (1+\epsilon^{-2} \vert w\vert^2)\right\},\eqlab{normG}
\end{align}
on the space
\begin{align*}
  \mathcal G_\epsilon:=\{Y:\Omega\rightarrow \mathbb C \mbox{ is analytic and } \Vert Y\Vert_\epsilon<\infty\}.
\end{align*}
The normed space $(\mathcal G_\epsilon,\Vert \cdot\Vert_\epsilon)$ is a complete space for each $\epsilon>0$, $\mathcal G_\epsilon\subset \mathcal G_\delta$ with $\Vert Y\Vert_\delta\le \Vert Y\Vert_\epsilon$ for all $0<\delta<\epsilon$, $Y\in \mathcal G_\epsilon$, and the factor $1+\epsilon^{-2} \vert w\vert^2$ in the norm $\Vert \cdot \Vert_\epsilon$ ensures that the convolution:
 \begin{align}
  (Y \star Z)(w) := \int_0^w Y(s) Z(w-s)ds=  w\int_0^1 Y(ws)Z(w(1-s))ds,\eqlab{convolutiondefn},
 \end{align}
 is continuous as a bilinear operator {$\mathcal G_\epsilon\times \mathcal G_\epsilon\rightarrow \mathcal G_\epsilon$} with the following estimate
\begin{align}
 \Vert Y\star Z \Vert_\epsilon \le 4\pi\epsilon \Vert Y\Vert_\epsilon \Vert Z\Vert_\epsilon\quad \forall\,Y,Z\in \mathcal G_\epsilon,\,\epsilon>0,\eqlab{convest}
\end{align}
 see \cite[Proposition 4]{bonckaert2008a}.

Moreover, the following result holds.
\begin{lemma}\cite[Proposition 3]{bonckaert2008a}\lemmalab{A1}
 The Laplace transform defines a linear continuous mapping, with operator norm $\Vert \mathcal L\Vert\le 1$, from $\mathcal G_\epsilon$ to the set of analytic functions on a local sector $\omega(\alpha,r)$ for $r>0$ sufficiently small. Moreover,  for every $Y,Z\in \mathcal G_\epsilon$ the following holds true:
 \begin{align}
  \mathcal L(Y \star Z)(x) &= \mathcal L(Y)(x)\mathcal L(Z)(x)\quad \forall\,x\in \omega(\alpha,r),\eqlab{convprop}
  \end{align}
  and
  \begin{align}
  x^2\frac{d}{dx}\mathcal L(Y)(x)&=\mathcal L(wY)(x)\quad \forall\,x\in \omega(\alpha,r),\eqlab{x2Lx}
 \end{align}
 with $wY$ being the function $w\mapsto wY(w)$.
\end{lemma}
In the following, {we  define $Y^{\star n}\in \mathcal G_\epsilon$ for all $Y\in \mathcal G_\epsilon$ and all $n\in \mathbb N$ recursively by $Y^{\star n}=Y \star Y^{\star (n-1)}$, \rsp{$n\ge 2$},  \rsp{and $Y^{\star 1}=1$}.}

 \begin{remark}
   The aforementioned results do not change in a significant way (only constants in estimates vary), if we were to replace the norm in \eqref{normG}
   by
   \begin{align*}
 \Vert Y\Vert_\epsilon=\sup_{w\in \Omega} \left\{ \vert Y(w)\vert e^{-\epsilon^{-1} \vert w\vert} (1+\epsilon^{-p} \vert w\vert^p)\right\},
 \end{align*}
 with $p>2$. This is not difficult to show; in fact, we will use a \rspp{similar} factor below that corresponds to $p=4$. 
 \end{remark}

\subsection{An equation in the Borel plane}
We consider \eqref{System0} and write $f_2$ as the convergent series:
\begin{align*}
 f_2(x,y) = \sum_{l=2}^\infty f_{2,l}(x) y^{l-2},
\end{align*}
where
\begin{align*}
 \sup_{x\in B_r(0)} \vert f_{2,l}(x)\vert \le AT^{l-1}\quad \forall\, l\ge 2,
\end{align*}
for some $A,T>0$. 
Let $F_0$ and $F_{2,l}$ denote the Borel-transforms $\mathcal B((\cdot)^2 f_0)$ and $\mathcal B((\cdot)^2 f_{2,l})$ of $x\mapsto x^{2} f_0(x)$ and $x\mapsto x^2 f_{2,l}(x)$, $l\ge 2$, respectively. These functions are entire functions of exponential growth. Moreover, $F_0(0)=F_{2,l}(0)=0$ for all $l\ge 2$. 

Then from \eqref{convprop} and \eqref{x2Lx}, one can set up the following equation for the Borel transform $\Phi=\Phi(w)$ of $y=\phi(x)$ satisfying \eqref{System0}:
\begin{align}\eqlab{Yeqn}
 (w+1)\Phi + a\star \Phi = F_0(w)+\sum_{l\ge 2} F_{2,l}(w) \star \Phi^{\star l}.
\end{align}
To be precise, we have the following from the discussion above: 
At the level of formal series, \eqref{System0} and \eqref{Yeqn} are equivalent, cf. $\mathcal L=\mathcal B^{-1}:\mathbb C[[w]]\rightarrow x\mathbb C[[x]]$. On the other hand, if $\Phi\in \mathcal G_\epsilon$ solves \eqref{Yeqn} then $\phi = \mathcal L(\Phi)$ solves \eqref{System0} on a local sectorial domain $\omega$, cf. \lemmaref{A1}.  

To solve \eqref{Yeqn} for $\Phi\in \mathcal G_\epsilon$, we can write it as a fix-point equation
\begin{align}
 \Phi(w) =  \frac{1}{w+1} \left( - a\star \Phi(w) +F_0(w)+\sum_{l\ge 2} F_{2,l}\star \Phi^{\star l}(w)\right).\eqlab{fp1}
\end{align}
This equation can easily be solved by Banach's fix-point theorem on a closed ball $0\le \Vert \Phi\Vert_\epsilon \le Q$ of $\mathcal G_\epsilon$, taking $Q>0$ large enough and $R$ in $\Omega(\alpha,R)$ small enough, see \eqref{Omega}, for all $0<\epsilon\ll 1$ upon using \eqref{convest}, see \cite{bonckaert2008a} and \cite[App. A]{uldall2024a} for an application of the result in the context of invariant manifolds only. 

\section{\rsp{An adaptation of the Borel-Laplace approach}}\seclab{adapt}

In this paper, we solve \eqref{Yeqn} on a different space (defined below), which provides further details of the singularity at $w=-1$. Whereas the form \eqref{fp1} (\rspp{used in \cite{bonckaert2008a}}) is based on inverting the first term in \eqref{Yeqn}, we will here need to invert the sum of the first two terms for $\Phi$.\footnote{It was Peter De Maesschalck that initially brought this idea to my attention while working on  a separate project.} More precisely, we consider the auxiliary equation
 \begin{align}
  (w+1)Y(w) +a\star Y(w) = H(w),\quad Y(0)=0,\eqlab{auxeqn}
 \end{align}
 and invert for $Y$. \rsp{The solution is described in \lemmaref{solYH} in \secref{new}. This leads to the definition of a new Banach space that controls the singularity of $Y$, see \eqref{normG2}. Subsequently, in \secref{prep} we state and prove some important lemmas on the behaviour of our norm. Here \lemmaref{convolution} and \lemmaref{YHest} on convolutions are particularly important. Then in \secref{solve} we write the equation \eqref{Yeqn} in a new fixed-point formulation (based upon \lemmaref{solYH} on inverting \eqref{auxeqn} for $Y$) that we solve using Banach's fixed point theorem and the estimates from \secref{prep}. This leads to an estimate of $\phi_n$ as $n\rightarrow \infty$ in \secref{init}. It is this estimate that forms the basis for the final proof of \thmref{main} in \secref{boot}.  }
\subsection{A new Banach space}\seclab{new}
We first describe the auxiliary equation \eqref{auxeqn}.
\begin{lemma}\lemmalab{solYH}
  Consider $a>1$ and any analytic function  $H$ with $H(0)=0$ defined on a domain $D\subset \mathbb C$ that is star-shaped at $0$. Then an analytic function
$Y:D\backslash (-\infty,-1]\rightarrow \mathbb C$ solves \eqref{auxeqn} 
 if and only if 
\begin{align}\eqlab{YH}
 Y(w) = \frac{H(w)}{1+w}-\frac{a}{(1+w)^{a+1}} \int_0^w H(s) (1+s)^{a-1} ds,\quad w\in D\backslash (-\infty,-1],
\end{align}
where the integration is taken along the line segment $[0,w]$.
\end{lemma}
\begin{proof}
 By differentiating the equation for $Y$ with respect to $w$, we obtain 
 \begin{align*}
  (w+1)Y' + (a+1) Y  = H',\quad Y(0)=0,\,w\in D\backslash(-\infty,-1].
 \end{align*}
 (Obviously, integrating with respect to $w$ along $[0,w]$ allows us to go the other way).
The linear initial value problem has the unique solution:
\begin{align}
 Y(w) = \frac{1}{(1+w)^{a+1}} \int_0^w H'(s) (1+s)^a ds,\quad w\in D \backslash (-\infty,-1],\eqlab{this}
\end{align}
with the integration being along the line segment $[0,w]$ (avoiding the potential branch cut along $(-\infty,-1]$).
Given that $a>1$ and $H(0)=0$, we obtain the desired form  through integration by parts. 
\end{proof}

We now restrict attention to $R=1$ in \eqref{Omega} and put \begin{align}
\Omega:= \Omega(\alpha,1) = S_\alpha\cup B_1(0),\eqlab{Omega2}
\end{align}
see \figref{omega}(a). Notice that the singularity $w=-1$ is now at the boundary of $\Omega$.
Then based upon \lemmaref{solYH}, we are led to define the following (refined) norm
\begin{align}
\Wert Y\Wert_\epsilon: = \sup_{w\in \Omega}\left\{\vert Y(w) \vert e^{-\epsilon^{-1} \vert w\vert} (1+\epsilon^{-4} \vert w\vert^{4}) \vert w\vert^{-1} \vert 1+w\vert^{a+1}\right\},\eqlab{normG2}
\end{align}
on the following (modified) space
\begin{align*}
\mathcal H_\epsilon:=\{Y:\Omega\rightarrow \mathbb C \mbox{  is analytic},\, Y(0)=0 \mbox{ and } \Wert Y\Wert_\epsilon<\infty\}.
\end{align*}
Here $\epsilon>0$ will be small enough. 
This new normed vector space $(\mathcal H_\epsilon,\Wert\cdot\Wert_\epsilon)$ is also complete for each $\epsilon>0$.
Now, whereas, the motivation for the factor $\vert 1+w\vert^{a+1}$ comes from \eqref{this} and the singularity at $w=-1$, the reason for the factor $(1+\epsilon^{-4} \vert w\vert^{4}) \vert w\vert^{-1}$, which differ from the one used in the norm of \cite{bonckaert2008a}, recall \eqref{normG}, will hopefully be clear from the following (see \remref{this}). 

Notice that functions in $\mathcal H_\epsilon$ can still be Laplace transformed. (In principle, this is not important for us since our focus is on the formal series; in fact, $\Omega=B_1(0)$ would suffice for the present paper.)  In particular, functions in $Y \in \mathcal H_\epsilon$ restrict to $Y\vert_{\Omega(\alpha,R)}:\Omega(\alpha,R)\rightarrow \mathbb C$ with $0<R<1$ that are bounded in the norm \eqref{normG}. Moreover, for all $0<\delta<\epsilon$ we have 
$\mathcal H_{\epsilon}\subset \mathcal H_\delta$ with 
\begin{align}
 \Wert Y\Wert_{\delta}\le C \Wert Y\Wert_\epsilon\quad \forall\,Y\in \mathcal H_\epsilon,\eqlab{GepsGdelta}
\end{align}
where 
\begin{align*}
C =\sup_{p\ge 0,r>1} \frac{e^{-r p }(1+r^4 p^4)}{e^{-p}(1+p^4)}.
\end{align*}
Here we have introduced $r= \delta^{-1}\epsilon$ and $p=\epsilon^{-1}\vert w\vert$.
This should be clear enough. To obtain that $C<\infty$, we first point out that it suffices to consider $p$ in a compact set of $p\in [0,b]$, $b>0$. Indeed, by direct differentiation, we obtain that the quantity 
\begin{align}
\frac{e^{-r p }(1+r^4 p^4)}{e^{-p}(1+p^4)},\quad p\ge 0,\,r>1,\eqlab{quantity}
\end{align} 
is decreasing with respect to $r$ if
\begin{align}
(r p)^4 - 4 (rp)^3 + 1>0.
\end{align}
Given that $r>1$, we have 
\begin{align*}
 (r p)^4 - 4 (rp)^3 + 1\ge r^3 p^3(p-4)+1,
\end{align*}
for $p\ge  0$. It therefore suffices to take $b=4$. On $p\in [0,4]$, the denominator of \eqref{quantity} is uniformly bounded from below by a constant $\underline C>0$; numerically we find that $\underline C\approx 0.6148$. Then to complete the proof of the claim, we use that the numerator of \eqref{quantity}, which is a function of $rp$, has a maximum; numerically we find that this maximum occurs at $rp\approx 3.984$ with value $\overline C\approx 4.707$. In total, $C\le  \underline C^{-1} {\overline C}\approx 7.656$.

\begin{remark}\remlab{resonance}
 \eqref{YH} and \eqref{this} illustrate why we work with a normal form with $a>1$. Notice in this regard that if $a\in -\mathbb N$ then we obtain logarithmic singularities in general at $w=-1$ of the Borel transform $Y$. To see this, take $H(w) = (1+w)^{-a}$ in \rspp{\eqref{YH}}. Then 
 \begin{align*}
  Y(w) = - a(1+w)^{-a-1} \log (1+w).
 \end{align*}
 If $\rspp{\phi=\sum_{n=1}^\infty \phi_n x^n=\mathcal B^{-1}(Y)(x)}\in x\mathbb C[[x]]$ then
a simple calculation shows that
\begin{align}\eqlab{fuck}
\frac{(-1)^n \phi_n}{\Gamma(n+a)} &= -a \sum_{k=0}^{-a-1} \begin{pmatrix} -a-1\\ k\end{pmatrix} \frac
{(-1)^{k}}{n-1-k} \frac{(n-1)!}{\Gamma(n+a)}\quad \forall\,n\ge N,
\end{align}
which at first glance does not appear to be bounded (as claimed in \thmref{main}); $$\frac{(n-1)!}{\Gamma(n+a)} = (n-1)\cdots (n+a)\rightarrow \infty\quad\mbox{for}\quad n\rightarrow \infty\quad \forall\,a\in -\mathbb N.$$ However, there is a certain cancellation of terms due to
\begin{align*}
 \sum_{k=0}^p\begin{pmatrix} p\\ k\end{pmatrix}\frac{(-1)^k}{q-k} = \frac{(-1)^p p! (q-p-1)!}{q!}\quad \forall\,q>p,
\end{align*}
which can be proven by induction on $p$. Indeed, upon using this property (with $p=-a-1$ and $q=n-1$) in \eqref{fuck}, we obtain
\begin{align*}
\frac{(-1)^n \phi_n}{\Gamma(n+a)} &= {a(-1)^{-a}} (-a-1)! \frac{(n+a-1)!}{\Gamma(n+a)}= a(-1)^{-a} (-a-1)! \quad \forall\,n\ge N,
\end{align*}
so that the quantity $S_\infty$ in \thmref{main} becomes 
\begin{align*}
 S_\infty =a(-1)^{-a} (-a-1)!
\end{align*}
Clearly, the corresponding differential equation is given by \eqref{findpy} with $a\in -\mathbb N$ and
\begin{align*}
f(x)=\mathcal B^{-1} (H)(x) = \sum_{k=0}^{-a+1} \frac{(-a)!}{(-a+1-k)!}x^k,\quad H(w)=(1+w)^{-a}. 
\end{align*}
By applying the normal form step in \secref{nf}, so that $a>1$, we avoid issues with nonintegrability of \eqref{YH} as well as logarithmic singularities. 
\end{remark}


In the following, \textbf{we drop the subscripts on $\mathcal H_\epsilon$ and $\Wert \cdot\Wert_\epsilon$}. 

\subsection{Preparatory lemmas}\seclab{prep}
Consider $\delta>0$ small enough and define $$\Omega_-(\delta) := B_\delta(-1)\cap \Omega\,\,\, \text{and}\,\,\, \Omega_+(\delta) := \Omega\backslash\Omega_-(\delta).$$ We will later fix $\delta=\epsilon^{2/3}$ (see the proof of \lemmaref{Tdefn}) but leave it general for now. 

\begin{lemma}\lemmalab{lowerbound1}
 Suppose that $w\in \Omega_-(\delta)$, $0<\delta<1$. Then
\begin{align*}
 \vert 1+w(1-t)\vert \ge\frac{1}{\sqrt{5}}  (\vert 1+w\vert +t )\quad \forall\, t\in (0,1).
\end{align*}

\end{lemma}
 \begin{proof}
  We put $w=-1+re^{i\theta}$, so that $r=\vert 1+w\vert$ and
  \begin{align}
   \vert 1+w(1-t)\vert^2 =r^2 (1-t)^2 +2 r t (1-t)\cos \theta+t^2.\eqlab{comp}
  \end{align}
  Since $w\in \Omega_-(\delta)$, it follows that $\theta\in(-\pi/2,\pi/2)$ and $r\in (0,\delta)$. Therefore
\begin{align*}
 \vert 1+w(1-t)\vert^2 \ge r^2 (1-t)^2 +t^2\ge \frac15 (r+t)^2\quad \forall\, t\in [0,1],\,r\in (0,\delta).
\end{align*}
To see this last estimate, we notice that the function $t\mapsto r^2 (1-t)^2 +t^2- \frac15 (r+t)^2$ has a minimum
at $t=\frac{r(1+5r)}{4+5r^2}\in (0,1)$ with positive value
\begin{align*}
 \frac{r^2(3+r)(1-r)}{4+5r^2}>0\quad \forall\,r\in (0,\delta),
\end{align*}
using here that $0<\delta<1$.
 \end{proof}
 \begin{lemma}\lemmalab{lowerbound2}
  Define
  \begin{align*}
   I(w):=\int_0^1 \frac{t^n}{\vert 1+w(1-t)\vert^{a+1}}dt,
  \end{align*}
where $w\in \Omega_-(\delta)$, $n\in \mathbb N_0$ and $n< a$. Then 
\begin{align*}
 I(w) \le 
           C_{a,n}\vert 1+w\vert^{n-a},
\end{align*}
where
\begin{align*}
 C_{a,n} = \frac{5^{\frac{a+1}{2}} \Gamma(a-n)\Gamma(n+1)}{\Gamma(a+1)}.
\end{align*}

 \end{lemma}
\begin{proof}
 We use \lemmaref{lowerbound1} to estimate the integral:
\begin{align*}
 I(w) \le 5^{\frac{a+1}{2}} \vert 1+w\vert^{n-a} \int_0^\infty \frac{s^n}{(1+s)^{a+1}}ds.
\end{align*}
Here we have used the substitution $t=\vert 1+w\vert s$ and estimated $\int_0^{\vert 1+w\vert^{-1}}(\cdots) ds\le \int_0^{\infty}(\cdots) ds$. 
The result then follows from 
\begin{align*}
 \int_0^\infty \frac{s^n}{(1+s)^{a+1}}ds = \frac{\Gamma(a-n)\Gamma(n+1)}{\Gamma(a+1)},
\end{align*}
cf. \eqref{integral} ($x=a-n,y=1+n$).

\end{proof}

\begin{lemma}\lemmalab{convolution}
 Suppose that $Y,Z\in \mathcal H$. Then
 \begin{align*}
  Y\star Z \in \mathcal H,
 \end{align*}
and there is a constant $\mathcal C>0$, that only depends upon $\alpha$ (recall \eqref{Omega2}) and $a>1$, such that
\begin{align}\eqlab{YZest1}
 \vert (Y\star Z)(w)\vert\le \mathcal C\Wert Y\Wert \Wert Z\Wert e^{\epsilon^{-1} \vert w\vert} (1+\epsilon^{-4}\vert w\vert^4)^{-1} \vert w\vert  \times &\begin{cases}
                 \epsilon^2 \vert 1+w\vert^{-a-1}\quad \forall\, w\in \Omega_+(\delta)\\
                 \vert 1+w\vert^{1-a} \quad \forall\,w\in \Omega_-(\delta),                                                                                                                                      \end{cases}
\end{align}
\rspp{for all $\epsilon>0$, $\delta>0$}.
In particular, 
\begin{align}\eqlab{YZest2}
 \Wert Y\star Z \Wert \le \mathcal C \epsilon^2 \Wert Y\Wert \Wert Z\Wert.
\end{align}
\end{lemma}
\begin{proof}
 We compute:
 \begin{align*}
  \vert (Y\star Z)(w)\vert&\le 2\Wert Y\Wert \Wert Z\Wert  e^{\epsilon^{-1} \vert w\vert} \vert w\vert \int_0^{\frac12} \frac{\vert w\vert(1-t)}{\vert 1+wt\vert^{a+1} }\frac{1}{1+\epsilon^{-4}\vert w\vert^4 (1-t)^4}  \\
  &\times \frac{\vert w\vert t}{ \vert 1+w(1-t)\vert^{a+1}} \frac{1}{1+\epsilon^{-4}\vert w\vert^4 t^4}dt,
 \end{align*}
 using \eqref{convolutiondefn} and symmetry of the integral around $s=\frac12$.
Clearly, since $\vert 1+z\vert\ge \frac12$ for all $z\in B_{\frac12}(0)$ we have that 
\begin{align}
 \vert 1+wt\vert \ge \tfrac12\quad \forall\,t\in [0,\tfrac12],\,w\in \Omega.\eqlab{simpleest}
\end{align}
Moreover, 
\begin{align*}
 {1+\epsilon^{-4}\vert w\vert^4 (1-t)^4}>\frac{1}{2^4} (1+\epsilon^{-4}\vert w\vert^4)\quad \forall\,t\in [0,\tfrac12],\,w\in \Omega.
\end{align*}
Hence we have that
\begin{align}
\vert (Y\star Z)(w)\vert&\le  2^{a+6}\Wert Y\Wert \Wert Z\Wert e^{\epsilon^{-1} \vert w\vert}  (1+\epsilon^{-4}\vert w\vert^4)^{-1}   \vert w\vert^2 \int_0^{\frac12} \frac{\vert w\vert t}{ \vert 1+w(1-t)\vert^{a+1}} \frac{1}{1+\epsilon^{-4}\vert w\vert^4 t^4}dt.\eqlab{est0conv}
\end{align}
We now divide the analysis into $w\in \Omega_-(\delta)$ and $w\in \Omega_+(\delta)$. Consider first $w\in \Omega_-(\delta)$. Then $\vert w\vert \le 1$ and by  \lemmaref{lowerbound2} (with $n=1$ and $a>1$) we have that 
\begin{align}
\vert (Y\star Z)(w)\vert&\le  2^{a+6} Y\Wert \Wert Z\Wert e^{\epsilon^{-1} \vert w\vert} (1+\epsilon^{-4}\vert w\vert^4)^{-1}  \vert w\vert \int_0^{\frac12} \frac{t}{ \vert 1+w(1-t)\vert^{a+1}} dt\nonumber \\
&\le  \overline C_{a}  Y\Wert \Wert Z\Wert  e^{\epsilon^{-1} \vert w\vert} (1+\epsilon^{-4}\vert w\vert^4)^{-1} \vert w\vert \vert 1+w\vert^{1-a},\eqlab{est1conv}
\end{align}
setting $\overline C_a:=C_{a,1} 2^{a+6}$. 
Next, for $w\in \Omega_+(\delta)$, we use that $\vert 1+w(1-t)\vert\ge \frac12 \vert 1+w\vert$ for all $t\in [0,1/2]$. This follows \rspp{from} \eqref{comp} with $r=\vert 1+w\vert$: 
$$\vert 1+w(1-t)\vert^2 \ge \vert 1+w\vert^2(1-t)^2\ge \frac14 \vert 1+w\vert^2.$$ Therefore we have that 
\begin{align}
\vert (Y\star Z)(w)\vert&\le  2^{2a+7} \Wert Y\Wert \Wert Z\Wert  e^{\epsilon^{-1} \vert w\vert}  (1+\epsilon^{-4}\vert w\vert^4)^{-1} \vert w\vert \vert 1+w\vert^{-a-1} \int_0^{\frac12} \frac{\vert w\vert t}{1+\epsilon^{-4}\vert w\vert^4 t^4}\vert w\vert dt\eqlab{est2conv} \\
&\le  2^{2a+7} \Wert Y\Wert \Wert Z\Wert  e^{\epsilon^{-1} \vert w\vert} (1+\epsilon^{-4}\vert w\vert^4)^{-1} \vert w\vert  \vert 1+w\vert^{-a-1} \epsilon^2 \int_0^{\infty } \frac{s}{1+s^4}ds,\nonumber
\end{align}
for all $w\in \Omega_+(\delta)$, 
using the substitution $s=\epsilon^{-1} \vert w\vert t$ and the  estimate $\int_0^{\frac12 \epsilon^{-1} \vert w \vert }(\cdots) ds\le \int_0^{\infty}(\cdots) ds$. 
Since $\int_0^{\infty } \frac{s}{1+s^4}ds=\pi/4<\infty$ this completes the proof of the estimate \eqref{YZest1}.

Finally, the estimate \eqref{YZest2} follows from \eqref{YZest1}. Indeed, we multiply both sides of \eqref{YZest1} by $e^{-\epsilon^{-1}\vert w\vert} (1+\epsilon^{-4} \vert w\vert^4) \vert w\vert^{-1}\vert 1+w\vert^{1+a}$ and let $\delta\rightarrow 0$. The result then follows by the definition of the norm $\Wert \cdot\Wert$ on the space $\mathcal H$, see \eqref{normG2}. 
\end{proof}

\begin{remark}\remlab{this}
We see the role of the $\vert w\vert^{-1}$-factor in the norm \eqref{normG2} in the estimates in \lemmaref{convolution}. Notice in particular, that in \eqref{est0conv} this factor leads to an integral
\begin{align*}
 \int_0^{\frac12} \frac{t}{\vert 1+w(1-t)\vert^{a+1}}dt,
\end{align*}
which can be estimated by \lemmaref{lowerbound2} ($n=1,a>1$). In this way, we obtain an upper bound with a singularity at $w=-1$ of the form $\vert 1+w\vert^{1-a}$. Notice that under convolution the singularity becomes ``less singular'', going from $\vert 1+w\vert^{-1-a}$ to $\vert 1+w\vert^{1-a}$. In essence, we gain a factor of $\vert 1+w\vert^2$. This improvement will be crucial in the next fundamental lemma, \lemmaref{YHest}. Without the $\vert w\vert^{-1}$-factor in the norm, we would not get this improvement. Notice at the same time, that the $\vert w\vert^{-1}$-factor leads to the integral 
\begin{align*}
\int_0^{\frac12} \frac{\vert w\vert t}{1+\epsilon^{-4}\vert w\vert^4 t^4}\vert w\vert dt
\end{align*}
in \eqref{est2conv}. This clearly demonstrates the necessity to replace $1+\epsilon^{-2}\vert w\vert^2$ in the norm of \cite{bonckaert2008a}, see \eqref{normG}. In particular, with this old factor we would have  
\begin{align*}
\int_0^{\frac12} \frac{\vert w\vert t}{1+\epsilon^{-2}\vert w\vert^2 t^2}\vert w\vert dt,
\end{align*}
instead,
which cannot be estimates uniformly in $\vert w\vert$ (with $\int_0^\infty \frac{s}{1+s^2}ds$ being divergent). 
\end{remark}

\begin{lemma}\lemmalab{YHest}
Consider $Y$ defined by \rspp{\eqref{YH}} and suppose that $H:\Omega\rightarrow \mathbb C\in \mathcal H$ is analytic and satisfies
\begin{align}\eqlab{estH}
 \vert H(w)\vert \le C e^{\epsilon^{-1} \vert w\vert} (1+\epsilon^{-4} \vert w\vert^{4})^{-1} \vert w\vert \times \begin{cases}
                  \epsilon^2 \vert 1+w\vert^{-a-1} & \forall\,w\in \Omega_+(\delta),\\
                  \vert 1+w\vert^{1-a} & \forall\,w\in \Omega_-(\delta),                                                                                                                     \end{cases}
\end{align}
for some $C>0$ 
for all $\epsilon>0$, $\delta>0$. 
Then there are constants $\overline C_1>0$ and $\overline C_2>0$ so that 
\begin{align*}
 \Wert Y\Wert \le C(\overline C_1\epsilon^2 \delta^{-2} + \overline C_2\delta),
\end{align*}
for all $0<\epsilon\ll 1$, $0<\delta<\frac12$.

\end{lemma}
\begin{proof}
From the linearity of \rspp{\eqref{YH}} with respect to $H$ it follows that it is without loss of generality to take $C=1$. From \rspp{\eqref{YH}}, we then directly obtain that
 \begin{align*}
  \vert Y(w)\vert e^{-\epsilon^{-1} \vert w\vert} (1+\epsilon^{-4} \vert w\vert^{4}) \vert w\vert^{-1}\vert 1+w\vert^{a+1}\le \mathcal T_1(w)+a \mathcal T_2(w), 
 \end{align*}
where
\begin{align*}
 \mathcal T_1(w)  &=  \vert  H(w) \vert e^{-\epsilon^{-1} \vert w\vert} (1+\epsilon^{-4} \vert w\vert^{4}) \vert w\vert^{-1}\vert 1+w\vert^{a} ,\\
 \mathcal T_2(w)  &= e^{-\epsilon^{-1} \vert w\vert} (1+\epsilon^{-4} \vert w\vert^{4})  \int_0^1 \vert  H(ws) \vert \vert 1+ws\vert^{a-1} ds.
\end{align*}
We estimate $\mathcal T_1(w)$ and $\mathcal T_2(w)$ successively using \eqref{estH}. For $\mathcal T_1(w)$, we obtain:
\begin{align*}
 \mathcal T_1(w) &\le \vert 1+w\vert^{a} \times \begin{cases}
                  \epsilon^2 \vert 1+w\vert^{-1-a} & \forall\,w\in \Omega_+(\delta),\\
                  \vert 1+w\vert^{1-a} & \forall\,w\in \Omega_-(\delta),                                                                                                                     \end{cases}\\
                 & \le \begin{cases}
                  \epsilon^2 \delta^{-1} & \forall\,w\in \Omega_+(\delta),\\
                  \delta & \forall\,w\in \Omega_-(\delta),   \end{cases}
\end{align*}
where we have used that
\begin{align*}
\vert 1+w\vert&\gtrless \delta \quad \mbox{for all}\quad w\in \Omega_\pm(\delta),\,\mbox{respectively}.
\end{align*}
Moreover,
\begin{align}
 \mathcal T_2(w) &\le \int_0^1 e^{-\epsilon^{-1} \vert w\vert(1-s) } \frac{1+\epsilon^{-4} \vert w\vert^{4}}{1+\epsilon^{-4} (\vert w\vert s)^{4}} \vert w\vert s \times \begin{cases}
                  \epsilon^2  \vert 1+ws\vert^{-2} &\forall\, ws\in \Omega_+(\delta),\\
                  1 &\forall\, ws\in \Omega_-(\delta).                                                                                                                    \end{cases}ds.\eqlab{T2esthere}
\end{align}
We first suppose that $w\in \Omega_+(\delta)$. Then $ws\in \Omega_+(\delta)$ for all $s\in [0,1]$ and we divide the resulting integral into two pieces, $s\in [0,\tfrac12]$ and $s\in [\tfrac12,1]$, i.e. 
\begin{align*}
 \int_0^1 (\cdots) ds = \int_0^{\frac12}(\cdots)ds+\int_{\frac12}^1(\cdots)ds.
\end{align*}
Then in the first integral, we obtain 
\begin{align*}
 \int_0^{\frac12} (\cdots) ds\le e^{-(\epsilon^{-1} \vert w\vert) \frac12 } (1+(\epsilon^{-1} \vert w\vert)^{4})  (\epsilon^{-1}\vert w\vert)                 \epsilon^3  \le  \epsilon^3 F,
\end{align*}
where $F=\sup_{A\ge 0}e^{-\frac12 A}(1+A^4)A <\infty$. Here we have also used the estimate $\vert 1+ws\vert\ge \frac12$ for all $s\in [0,\tfrac12]$, $w\in \Omega_+(\delta)$, recall \eqref{simpleest}. 
Next, for the remaining integral from $s=\frac12$ to $s=1$ we obtain
\begin{align*}
\int_{\frac12}^1 (\cdots) ds\le  \epsilon^2 \int_{\frac12}^1 \frac{\vert w\vert s}{\vert 1+ws\vert^2}ds \frac{1+\epsilon^{-4} \vert w\vert^{4}}{1+\epsilon^{-4} (\vert w\vert \frac12)^{4}}.
\end{align*}
It is easy to show that 
\begin{align*}
 \frac{ \vert w\vert s}{\vert 1+ws\vert^{2}}\le \delta^{-2}\quad \forall\,s\in [\tfrac12,1],\,w\in \Omega_+(\delta),
\end{align*}
for all $0<\delta<\frac12$. Indeed, it clearly holds for $\vert w\vert s \le 1$ whereas for $\vert w\vert s>1$ we have $\operatorname{Re}(w s)>0$ (given that $w\in \Omega_+(\delta)$, recall \eqref{Omega2})  and  therefore $\vert 1+ws\vert \ge \vert ws\vert$. Consequently, 
\begin{align*}
 \frac{ \vert w\vert s}{\vert 1+ws\vert^{2}}\le \frac{1}{\vert w\vert s}< 1.
\end{align*}
Moreover, 
\begin{align}
 \frac{1+z}{1+z 2^{-4}}\le 2^4\quad \forall\,z\ge 0.\eqlab{Aest1}
\end{align}
Collectively, we have 
\begin{align*}
 \mathcal T_2(w)\le  \overline C_1 \epsilon^2 \delta^{-2}\quad \forall\, w\in \Omega_+(\delta),
 \end{align*}
 for $\overline C_1>0$ large enough. 
 
 Next, we turn to the case where $w\in \Omega_-(\delta)$ in \eqref{T2esthere} (in which case $\vert w\vert \le 1$). Then
\begin{align*} 
\mathcal T_2(w)&\le  \int_{s:ws\in \Omega_+(\delta)} e^{-\epsilon^{-1} \vert w\vert(1-s) } \frac{1+\epsilon^{-4} \vert w\vert^{4}}{1+\epsilon^{-4} (\vert w\vert s)^{4}} \vert w\vert s     \epsilon^2 \vert w\vert  \vert 1+ws\vert^{-2} ds \\
&+ \int_{s:ws\in \Omega_-(\delta)} e^{-\epsilon^{-1} \vert w\vert(1-s) } \frac{1+\epsilon^{-4} \vert w\vert^{4}}{1+\epsilon^{-4} (\vert w\vert s)^{4}} \vert w\vert s  ds. 
\end{align*}
For the first integral, we can proceed as above for $w\in \Omega_+(\delta)$ and estimate $\overline C_1\epsilon^2\delta^{-2}$. For the second integral, we substitute $s=1-t$ and find by simple geometry that $$\mbox{$(w\in \Omega_-(\delta)$ and $w(1-t)\in \Omega_-(\delta))$ $  \Rightarrow 0\le t\le \delta $}, $$ in the last integral. Then, since 
\begin{align*}
 \frac{1+\epsilon^{-4} \vert w\vert^{4}}{1+\epsilon^{-4} (\vert w\vert (1-t))^{4}}\le 2^4, 
\end{align*}
for all $t\in [0,\delta]$ with $0<\delta<\frac12$, recall \eqref{Aest1}, we have
\begin{align*} 
\mathcal T_2(w)&\le \overline C_1\epsilon^2 \delta^{-2} + \overline C_2\delta,
\end{align*}
for $\overline C_2>0$ large enough.  This completes the proof. 
\end{proof}
\begin{lemma}\lemmalab{Fest}
 Consider a convergent series $f=\sum_{k=2}^\infty f_k x^{\rsp{k}} \in x^2 \mathbb C\{x\}$, $\vert f_k\vert\le A T^{k-1}$ and let $F(w)=\mathcal B(f)(w)=\sum_{k=2}^\infty \frac{f_k}{(k-1)!}w^{k-1}$ denote the Borel transform of $f$. Then $F\in \mathcal H$ for all $0<\epsilon\ll 1$. In particular, there is an $\epsilon_0=\epsilon_0(a,T)>0$ small enough and a $K=K(a)>0$ large enough such that 
 \begin{align*}
  \Wert F\Wert\le  K A T,
 \end{align*}
 for all $0<\epsilon< \epsilon_0$.
\end{lemma}
\begin{proof}
 We have 
 \begin{align*}
  \vert F(w)\vert \le A \sum_{k=2}^\infty \frac{T^{k-1} \vert w\vert^{k-1}}{(k-1)!} \le A(e^{T\vert w\vert }-1)\le AT \vert w\vert e^{T\vert w\vert}.
 \end{align*}
The last equality is due to $\int_0^{T \vert w\vert} e^{s} ds \le T\vert w\vert e^{T\vert w\vert}$. Therefore
\begin{align*}
 \Wert F\Wert &\le  \sup_{w\in \Omega} \left(A T e^{(T-\epsilon^{-1})\vert w\vert} (1+\epsilon^{-4} \vert w\vert^4) \vert 1+w\vert^{a+1}\right)\\
 &\le A T \sup_{t\ge 0} \left( e^{(T-\epsilon^{-1})t} (1+\epsilon^{-4} t^4) \vert 1+t\vert^{a+1}\right)\\
 &=A T \sup_{t\ge 0} \left( e^{-t} (1+t^4) e^{\epsilon T t} \vert 1+\epsilon t\vert^{a+1}\right),
\end{align*}
where we have used that $a+1>0$ and replaced $t$ by $\epsilon t$ in the last equality. 
The function $t\mapsto e^{-t} (1+t^4) e^{\epsilon T t} \vert 1+\epsilon t\vert^{a+1}$ is clearly bounded with respect to $t\ge 0$ for any $\epsilon^{-1}>T$. 
The result now follows from the fact that $e^{-t_2}(1+t_2^4)$, $t_2\ge 0$, has a single maximum (around $t_2\approx 3.984$ with value $\approx 4.707$). 
\end{proof}

\begin{lemma}\lemmalab{thisfinal}
Consider $Y$ defined by \rspp{\eqref{YH}} and suppose that $H:\Omega\rightarrow \mathbb C$ analytic satisfies
\begin{align*}
 \vert H(w)\vert \le AT \vert w\vert e^{T\vert w\vert}.
\end{align*}
Then there is an $\epsilon_0=\epsilon_0(a,T)>0$ small enough and a $K=K(a)>0$ large enough such that 
\begin{align*}
 \Wert Y\Wert \le KAT\quad \forall\,\epsilon\in (0,\epsilon_0).
\end{align*}
\end{lemma}
\begin{proof}
We estimate
\begin{align*}
  \vert Y(w)\vert e^{-\epsilon^{-1} \vert w\vert} (1+\epsilon^{-4} \vert w\vert^{4}) \vert w\vert^{-1}\vert 1+w\vert^{a+1}\le AT (\mathcal T_1(w)+a \mathcal T_2(w)), 
 \end{align*}
 where
 \begin{align*}
  \mathcal T_1(w)&=e^{(T-\epsilon^{-1}) \vert w\vert } (1+\epsilon^{-4}\vert w\vert^4) \vert 1+w\vert\\
  &\le \sup_{t\ge 0} e^{-t} (1+t^4) e^{\epsilon T t}(1+\epsilon t),
 \end{align*}
 and 
 \begin{align*}
 \mathcal T_2(w) &= \int_0^1 \vert w\vert s e^{T\vert w\vert s}\vert 1+s w\vert^{a-1} ds e^{-\epsilon^{-1} \vert w\vert} (1+\epsilon^{-4} \vert w\vert^{4})\\
  &\le \vert w\vert e^{(T-\epsilon^{-1}) \vert w \vert } (1+\vert w\vert)^{a-1} (1+\epsilon^{-4} \vert w\vert^{4})\\
  &\le \epsilon \sup_{t\ge 0} t e^{-t}  (1+t^4)  e^{\epsilon Tt}(1+\epsilon t)^{a-1}.
 \end{align*}
Here we have used that $a>1$. Proceeding as in the proof of \lemmaref{Fest}, the result now follows. 
\end{proof}

\subsection{Solving the equation}\seclab{solve}
 We are now in a position to solve \eqref{Yeqn} in $\mathcal H$ through a fix-point argument. Let $B_Q^{\mathcal H}\subset \mathcal H$ denote the ball of radius $Q>0$ centered at the origin in $\mathcal H$
and define the operator $\mathcal T:B_Q^{\mathcal H}\rightarrow B_Q^{\mathcal H}$ through \eqref{YH}:
\begin{align*}
 \mathcal T(Y)(w)&= \frac{F_0(w)}{1+w} -\frac{a}{(1+w)^{a+1}}\int_0^w F_0(s)(1+s)^{a-1}ds\\
 &+\frac{1}{1+w} (Y\star \sum_{l=2}^\infty F_{2,l} \star Y^{\star (l-1)})(w)-\frac{a}{(1+w)^{a+1}}\int_0^w (Y\star \sum_{l=2}^\infty F_{2,l} \star Y^{\star (l-1)})(s)(1+s)^{a-1} ds.
 \end{align*}
 We emphasize that \eqref{Yeqn} with $Y\in B_Q^{\mathcal H}$ is equivalent with $Y=\mathcal T(Y)$, $Y\in B_Q^{\mathcal H}$, cf. \lemmaref{solYH}.
The next lemma, \textbf{which uses $\delta=\epsilon^{2/3}$}, shows that $\mathcal T$ is well-defined.  
\begin{lemma}\lemmalab{Tdefn}
 There is a $Q>0$ large enough and an $\epsilon_0=\epsilon_0(Q)>0$, such that $\mathcal T:B_Q^{\mathcal H}\rightarrow B_Q^{\mathcal H}$ is well-defined for all $0<\epsilon<\epsilon_0$.
\end{lemma}
\begin{proof}
For $Y\in B_Q^{\mathcal H}\subset \mathcal H$, we split $\mathcal T(Y)(w)$ into two terms:
\begin{align*}
 \mathcal T_1(w)=\frac{F_0(w)}{1+w} -\frac{a}{(1+w)^{a+1}}\int_0^w F_0(s)(1+s)^{a-1}ds,
\end{align*}
and
\begin{align*}
 \mathcal T_2(w) = \frac{1}{1+w} (Y\star \sum_{l=2}^\infty F_{2,l} \star Y^{\star (l-1)})(w)-\frac{a}{(1+w)^{a+1}}\int_0^w (Y\star \sum_{l=2}^\infty F_{2,l} \star Y^{\star (l-1)})(s)(1+s)^{a-1} ds.
\end{align*}
We first notice that $\mathcal T_1\in \mathcal H$ by \lemmaref{thisfinal}. Let $P=\Wert \mathcal T_1\Wert$ denote its norm, which is independent of $Y$. Next, for $\mathcal T_2$ we apply
\lemmaref{YHest}. For this purpose, we show that $H=Y\star \sum_{l=2}^\infty F_{2,l} \star Y^{\star (l-1)}$ satisfies the estimate \eqref{estH}. But this follows from \lemmaref{convolution}, see \eqref{YZest1}, once we have shown that $\sum_{l=2}^\infty F_{2,l} \star Y^{\star (l-1)}\in \mathcal H$:

By \lemmaref{Fest}, we have $\Wert F_{2,l}\Wert \le A T^{l-1}$ for some (new) $A,T>0$ and therefore by \lemmaref{convolution}, see \eqref{YZest2}, we have
\begin{align*}
\Wert \sum_{l=2}^\infty F_{2,l} \star Y^{\star (l-1)}\Wert \le \sum_{l=2}^\infty (\mathcal C\epsilon^2)^{l-1} A T^{l-1} \Wert Y\Wert^{l-1}\le A \sum_{l=1}^\infty ( \mathcal C T Q \epsilon^2 )^{l} \le  2A \mathcal C T Q \epsilon^2,
\end{align*}
for all $0<\epsilon<(2\mathcal C TQ)^{-\frac12}$. Given that we take $\delta=\epsilon^{2/3}$, it now follows that
\begin{align*}
 \Wert \mathcal T_2\Wert \le 2A \mathcal C T Q \epsilon^2 (\overline C_1\epsilon^2\delta^{-2} + \overline C_2\delta)= 2A \mathcal C T Q \epsilon^{8/3} (\overline C_1 + \overline C_2).
\end{align*}
We conclude that $\mathcal T:B_{Q}^{\mathcal H}\rightarrow B_Q^{\mathcal H}$ with $Q=2P$ is well-defined for all $0<\epsilon\ll 1$. 

\end{proof}

\begin{proposition}\proplab{Ysol}
 Let $Q=2P>0$ be as in (the proof of) \lemmaref{Tdefn}. Then $\mathcal T:B_Q^{\mathcal H}\rightarrow B_Q^{\mathcal H}$ is a contraction for all $0<\epsilon\ll 1$. 
 \end{proposition}
 \begin{proof}
Having already established that $\mathcal T:B_Q^{\mathcal H}\rightarrow B_Q^{\mathcal H}$ is well-defined for all $0<\epsilon\ll 1$, we have to bound the Lipschitz constant. For this we proceed as in \cite{bonckaert2008a} and repeat the estimates of \lemmaref{Tdefn} to $D\mathcal T(Y)(Z)$:
This leads to
\begin{align*}
 \Wert D\mathcal T(Y)(Z)\Wert =\mathcal O(\epsilon^{8/3}) \Wert Z\Wert\quad \forall\,Y\in B_Q^{\mathcal H},\,Z\in \mathcal H,
\end{align*}
using $\delta=\epsilon^{2/3}$. 
 \end{proof}
\subsection{Initial estimate of $\phi_n$}\seclab{init}
  We now fix $\epsilon>0$ small enough, and denote the unique fix-point of $\mathcal T$ in $B_Q^{\mathcal H}$ (\rspp{its existence} follows from Banach's fix-point theorem) by $\Phi$. 
 Define
 \begin{align*}
  Z(w) = \Phi(w)(1+w)^{a+1}.
 \end{align*}
 It follows that  $Z(w)$ is analytic and uniformly bounded on $B_1(0)$ (recall that $\epsilon>0$ is fixed). Consequently, by Cauchy's integral formula, we have that  $Z(w)=\sum_{n=1}^\infty Z_n w^n$, $0\le \vert w\vert<1$, with $\vert Z_n\vert\le C_0$ for some $C_0>0$ large enough. In turn, by Cauchy's product formula and the analytic expansion
\begin{align*}
 (1+w)^{-1-a} = \sum_{k=0}^\infty \frac{(-1)^k \Gamma(a+1+k)}{\Gamma(k+1)\Gamma(a+1)} w^k,\quad w\in B_1(0),
\end{align*}
we have that $\Phi(w)=\sum_{n=1}^\infty \Phi_n w^n$ with
\begin{align}
 \Phi_n = \sum_{k=0}^{n-1} Z_{n-k} \frac{(-1)^k \Gamma(a+1+k)}{\Gamma(k+1)\Gamma(a+1)}.\eqlab{PhinZ}
\end{align}
This directly leads to the following.
\begin{lemma}\lemmalab{borellaplaceest}
Let $\phi=\sum_{n=2}^\infty \phi_n x^n\in x^2 \mathbb C[[x]]$ denote the formal series solution of \eqref{system0}. Then there is a $C>0$ such that 
the following estimate holds true:
 \begin{align}
  \vert \phi_n\vert \le C \Gamma(a+n+1)\quad \forall\, n\ge 2.\eqlab{phinest}
 \end{align}

\end{lemma}
\begin{proof}
The proof follows from simple calculations: Indeed, from  \eqref{PhinZ} and the previous discussion, we have that 
 \begin{align*}
  \vert \Phi_n\vert \le \frac{C_0}{\Gamma(a+1)} \sum_{k=0}^{n-1} \frac{\Gamma(a+1+k)}{\Gamma(k+1)}.
 \end{align*}
The quantity $p(k):=\frac{\Gamma(a+1+k)}{\Gamma(k+1)}$ is increasing with respect to $k$:
\begin{align*}
 p'(k) = \frac{\Gamma(a+1+k)}{\Gamma(k+1)}\left(\psi(a+1+k)-\psi(1+k)\right)>0,
\end{align*}
since $a>0$ (recall \eqref{digamma}) and 
\begin{align}
 p(k) =(1+o_{k\rightarrow \infty}(1))k^{a}.\eqlab{estgamma}
\end{align}
by \eqref{stirling}.
Fix $n_0\gg 1$ and consider any $n\ge n_0$. We can then estimate
\begin{align*}
 \vert \Phi_n\vert \le C_1 + C_2 \int_{n_0}^{n} k^{a} dk\le C_3 n^{a+1}\le 2C_3 \frac{\Gamma((a+1)+n+1)}{\Gamma(n+1)} \quad\forall\,n\ge n_0,
\end{align*}
using \eqref{estgamma},
where $C_i=C_i(a,n_0)>0$, $i=1,\ldots,3$. In total, there is a $C_4=C_4(a,n_0)>0$ large enough such that 
\begin{align}\eqlab{Yn_1Est}
\vert \Phi_{n-1}\vert\le C_4\frac{\Gamma(a+1+n)}{\Gamma(n)} = C_4 \frac{\Gamma(a+1+n)}{(n-1)!} \quad \forall\, n\ge 2.
\end{align}
Now, recall that 
\begin{align*}
\phi = \mathcal B^{-1}\left(\sum_{n=2}^\infty \Phi_{n-1} w^{n-1}\right) = \sum_{n=2}^\infty (n-1)! \Phi_{n-1} x^n\in x^2\mathbb C[[x]],
\end{align*}
so that $\phi_n=(n-1)! \Phi_{n-1}$. We therefore obtain the estimate 
\begin{align*}
 \vert \phi_n\vert \le C_4 \Gamma(a+1+n)\quad \forall \,n\ge 2,
\end{align*}
from \eqref{Yn_1Est}. This completes the proof. 
\end{proof}

\section{Completing the proof of \thmref{main}}\seclab{boot}
In order to complete the proof of \thmref{main}, we now apply a bootstrapping step: Let $\phi=\sum_{x=2}^\infty \phi_n x^n\in x^2 \mathbb C[[x]]$ denote the formal series expansion of the center manifold of \eqref{System0} and write 
\begin{align*}
 f(x,y)  :=x^2 f_0(x)+x^2 y^2 f_2(x,y).
\end{align*}
We then consider the composition $f(x,\phi(x))\in x^2 \mathbb C[[x]]$ as a formal series and define $p_n=p_n(\phi)$ by
 \begin{align}
 f(x,\phi(x))=:\sum_{k=2}^\infty p_n x^n. \eqlab{pn}
 \end{align}
 The coefficients $p_n$ can be expressed in terms of the coefficients $\phi_n$ of $\phi$ upon expanding $f_0$ and $f_2$ in convergent power series and by using Cauchy's product formula, see \eqref{pnexpr} below.
 
From \lemmaref{mk0linearcase}, we conclude that $\phi_n$ satisfies
\begin{align}
 \phi_n = (-1)^n \Gamma(n+a) S_n\quad\mbox{where}\quad S_n:=\sum_{j=2}^n \frac{(-1)^j p_j}{\Gamma(j+a)}\quad \forall\,n\ge 2.\eqlab{fixpointphin}
\end{align}
In essence this gives a fix-point formulation of $\phi_n$ and this was used in \cite{euler} and \cite{MR4445442} to prove \thmref{main} under further assumptions. The following proposition uses \eqref{fixpointphin} to bootstrap the estimate \eqref{phinest}, repeated here for convenience:
\begin{align}
 \vert \phi_n\vert \le C \Gamma(n+1+a)\quad \forall\,n\ge 2.\eqlab{phiest2}
\end{align}
\begin{proposition}\proplab{final}
Suppose that \eqref{phiest2} holds true.
Then the series $S_\infty:=\sum_{j=2}^\infty \frac{(-1)^j p_j}{\Gamma(j+a)}$ is absolutely convergent and
\begin{align*}
 \frac{(-1)^n \phi_n}{\Gamma(n+a)}\rightarrow S_\infty.
 \end{align*}

\end{proposition}
\begin{proof}
Given \eqref{phiest2}, we will show that that there exists a $K>0$ such that 
\begin{align}
 \vert p_n\vert \le K \Gamma(n+a-3)\quad \forall\, n\ge 2,\eqlab{proppn}
\end{align}
recall \eqref{pn}.
Notice that \eqref{proppn} together with \eqref{fixpointphin} implies the result  since $$\vert S_\infty\vert \le \overline K\sum_{j=2}^\infty \frac{1}{j^3}<\infty,$$ for $\overline K>0$ large enough. Here we have used \eqref{stirling} \rspp{(with $x=j+a>j$ and $b=-3$)}.

To show \eqref{proppn}, we follow \cite[\rsp{Section 4}]{euler} and write $f$ as a convergent series
\begin{align*}
f(x,y) = \sum_{n=2}^\infty f_{0,n} x^n+\sum_{l=2}^\infty \sum_{n=2}^\infty f_{2,n,l} x^n y^l,
\end{align*}
where
\begin{align}
 \vert f_{0,n}\vert \le A T^{n-2},\quad \vert f_{2,n,l}\vert \le A T^{n+l-4},\quad \forall\,n\ge 2,l\ge 2,\eqlab{f02nbound}
\end{align}
for some $A,T>0$. 
Then it is a simple calculation, using Cauchy's product formula, to show the following: $p_n=f_{0,n},\forall\,n=2,\ldots,5$ and
\begin{align}
 p_n = f_{0,n}+\sum_{l=2}^{\lfloor \frac{n-2}{2}\rfloor} \sum_{k=2l}^{n-2} f_{2,n-k,l} (\phi^l)_k\quad \forall\, n\ge 6,\eqlab{pnexpr}
\end{align}
recall \eqref{pn}.
Here $(\phi^l)_n$ is defined by 
\begin{align*}
 \phi(x)^l =:\sum_{n=2l}^\infty (\phi^l)_n x^n \in x^{2l}\mathbb C[[x]]\quad \forall\,l\ge 2.
\end{align*}
For simplicity, we now write
\begin{align}\eqlab{wkdefn}
\Gamma_n: =\Gamma(n+1+a)\quad \forall\, n\ge 2,
\end{align}
and claim that 
\begin{align}
 \vert (\phi^l)_n\vert \le C^l \overline C^{l-1} \Gamma_{n-2(l-1)}\quad \forall\,n\ge 2l.\eqlab{philnest}
\end{align}
Here $\overline C=\overline C(a+1)$ and $C$ are the constants from \eqref{convolutionwkest} and \eqref{phiest2}, respectively.
The proof of \eqref{philnest} is by induction on $l$ with $l=2$ being the base case, where
\begin{align*}
 \vert (\phi^2)_n\vert\le C^2 \sum_{k=2}^{n-2} \Gamma_{k}\Gamma_{n-k}\le C^2 \overline C \Gamma_{n-2},
\end{align*}
using \eqref{convolutionwkest}, \eqref{phiest2} and Cauchy's product formula. For the induction step, we suppose that \eqref{philnest} holds true for $l\rightarrow l-1$ and write $ (\phi^l)_n = \sum_{k=2(l-1)}^{n-2} (\phi^{l-1})_k \phi_{n-k}$ so that
\begin{align*}
 \vert (\phi^l)_n\vert&\le C^{l}\overline C^{l-2} \sum_{k=2(l-1)}^{n-2} \Gamma_{k-2(l-2)} \Gamma_{n-k}\\
 &= C^{l}\overline C^{l-2} \sum_{k=2}^{n-2(l-1)} \Gamma_{k} \Gamma_{n-2(l-2)-k}\\
 &\le C^l \overline C^{l-1} \Gamma_{n-2(l-1)},
\end{align*}
as desired. Here we have used \eqref{convolutionwkest} and \eqref{phiest2} again.

We may assume that $\overline C>1$ and therefore put $K=C \overline C$ so that 
\begin{align*}
 \vert (\phi^l)_n\vert \le K^l \Gamma_{n-2(l-1)}\quad \forall\,n\ge 2l.
\end{align*}
We now estimate $p_n$, $n\ge 6$, in \eqref{pnexpr}. We have 
\begin{align*}
\vert p_n\vert \le A T^{n-2} +AK^{2} \sum_{l=2}^{\lfloor \frac{n-2}{2}\rfloor}(KT)^{l-2} \sum_{k=2l}^{n-2} T^{n-2-k} \Gamma_{k-2(l-1)},
\end{align*}
using \eqref{f02nbound} and \eqref{philnest}. We then use \eqref{rhowkest} to estimate the inner sum
\begin{align*}
 \sum_{k=2l}^{n-2}  T^{n-2-k} \Gamma_{k-2(l-1)} = \sum_{k=2}^{n-2l} T^{n-2l-k} \Gamma_{k}\le \overline C \Gamma_{n-2l}\quad \forall \,n\ge 2l+2,
\end{align*}
where $\overline C=\overline C(a,T)>0$. 
Consequently, 
\begin{align*}
\vert p_n\vert \le A T^{n-2} +A\overline C K^2 \sum_{l=2}^{\lfloor \frac{n-2}{2}\rfloor}(KT)^{l-2} \Gamma_{n-2l}.
\end{align*}
The remaining sum on the right hand side can then be estimated by \eqref{xiwkest}:
\begin{align*}
\vert p_n\vert \le A T^{n-2} +AK^{2} \overline C^2  \Gamma_{n-4}\quad \forall\,n\ge 6.
\end{align*}
upon increasing $\overline C>0$ if necessary. Therefore \eqref{proppn} follows upon using \eqref{wkdefn}. This completes the proof. 
\end{proof}
The property \eqref{mnasymp}  now follows from \propref{final} and this therefore completes the proof of \thmref{main}.

\section{\rsp{Analytic center manifolds for Riccati equations}}\seclab{riccati}
\rsp{In this section, we consider an application of our results to the family of Riccati equations:
\begin{align}
 x^2 \frac{dy}{dx} +(1+a x) y = b x^2+ y^2,\eqlab{riccati}
\end{align}
with $a,b\in \mathbb R$. We are interested in analytic center manifolds.
\begin{remark}
 Notice in comparison with \eqref{riccati0} that we have put $c=1$ in \eqref{riccati}. The reason is the following: The case $c=0$ in \eqref{riccati0} is covered by \lemmaref{mk0linearcase} and we therefore only have to consider $c\ne 0$. But in this case, we can reduce \eqref{riccati0} to \eqref{riccati} by setting $b=c^{-1} \tilde b$ and $y=c^{-1} \tilde y$ and then removing the tildes.  
\end{remark}

The following result provides a local characterization of analytic center manifolds of \eqref{riccati} with $(a,b)$ in a neighborhood of $(-2,0)$.

\begin{proposition}\proplab{riccati}
 Consider \eqref{riccati} with parameters $(a,b)\in \mathbb R^2$. Then there is a neighborhood $U\subset \mathbb R^2$ of 
 \begin{align}
  (a,b) = (-2,0),\eqlab{tr}
 \end{align}
 in the $(a,b)$-parameter space, 
 such that the formal series expansion of the center manifold $\phi=\sum_{n=2}^\infty \phi_n x^n \in x^2 \mathbb C[[x]]$ of \eqref{riccati} with $(a,b)\in U$, is convergent if and only if either of the following conditions hold:
 \begin{enumerate}
  \item \label{item1} $b=0$;
  \item \label{item2}
  \begin{align}
 a = -b-2.\eqlab{f11}
\end{align}
  
 \end{enumerate}
 
\end{proposition}

The result is illustrated in \figref{riccati}. We prove \propref{riccati} in \secref{riccati_proof}. Subsequently, in \secref{riccati_num}, we compute $S_\infty$ numerically and make additional (global) observations on analytic center manifolds of \eqref{riccati}. 

\subsection{Proof of \propref{riccati}}\seclab{riccati_proof}
The idea of the proof is to study the equation $S_\infty(a,b)=0$. Here $S_\infty$ is the quantity from \thmref{main}. 
Clearly, 
if $b=0$ then $y(x)=0$ is an analytic center manifold. Moreover, if \eqref{f11} holds then $\phi$ is a geometric series:
\begin{align*}
 \phi(x) =  \sum_{n=2}^\infty b^{n-1} x^n = \frac{b x^2}{1-b x},
\end{align*}
convergent for all $0\le \vert x\vert <\vert b\vert^{-1}$.
To see this, we can either insert $\phi(x)$ directly into \eqref{riccati} or alternatively we can insert
\begin{align}
 \phi_n = b^{n-1},\quad n\ge 2,\eqlab{phin2}
\end{align}
into the recursion relation: 
\begin{align}\eqlab{phineqn}
\phi_n +(n-1+a)\phi_{n-1}  &=p_n,\quad 
 p_n :=b\delta_{n2}+ \sum_{i=2}^{n-2} \phi_i \phi_{n-i},
\end{align}
see \eqref{recmk0}, with $a$ given by \eqref{f11}.
Here $p_n$ is defined by $f(x,\phi(x))=:\sum_{n=2}^\infty p_n x^n\in x^2\mathbb C[[x]]$ (as in \secref{boot}) with $f(x,y)=b x^2+y^2$. In particular, the expression in \eqref{phineqn} follows from Cauchy's product formula:
\begin{align*}
 \phi(x)^2 = \sum_{n=4}^\infty \left(\sum_{i=2}^{n-2} \phi_i \phi_{n-i}\right) x^n\in x^4 \mathbb C[[x]].
\end{align*}
Moreover, $\delta_{ij}$ denotes Kronecker's delta:
\begin{align*}
 \delta_{ij} = \begin{cases}
                1 & i=j\\
                0 & i\ne j               \end{cases}\quad i,j\in \mathbb N.
\end{align*}
This proves the if part of \propref{riccati}. 

\begin{figure}[h!]
\begin{center}
{\includegraphics[width=.4\textwidth]{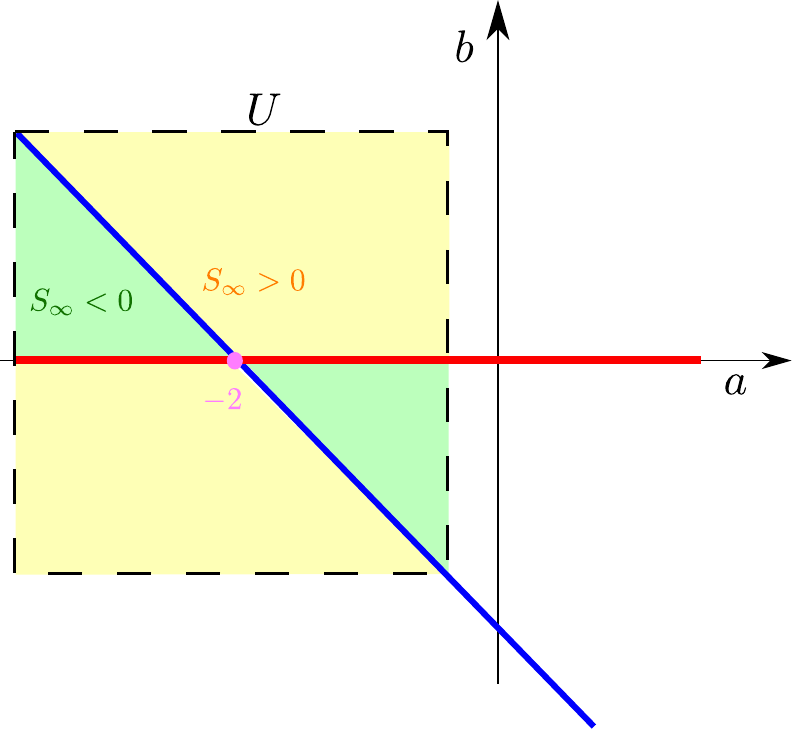}}
\end{center}
\caption{\rsp{Illustration of the result in \propref{riccati}. In a neighborhood of the point \eqref{tr}, $S_\infty^{-1}(0)$ is the union of the two intersecting branches in red and blue (corresponding to items \ref{item1} and \ref{item2} of \propref{riccati})}.}
\figlab{riccati}
\end{figure}

Working locally near \eqref{tr}, we may assume that $a>-3$. Then we can use \eqref{Sinfty0} to write $S_\infty=S_\infty(a,b)$, $(a,b)\in U$, as the uniformly  convergent infinite series:
\begin{equation}\eqlab{Sinfty1}
\begin{aligned}
 S_\infty& = \sum_{n=3}^{\infty} \frac{(-1)^n (-b(a+2)\delta_{n3} + p_n)}{\Gamma(n+a)}\\
 &=\sum_{n=3}^{\infty} \frac{(-1)^n (-b(a+2)\delta_{n3} + \sum_{i=2}^{n-2} \phi_i \phi_{n-i})}{\Gamma(n+a)}.
\end{aligned}
\end{equation}
Here we have used  \eqref{Gamma} and $p_2=b$ in the first equality and \eqref{phineqn} to express $p_n$, $n\ge 3$, in terms of $\phi_n=\phi_n(a,b)$ in the second equality. 
Moreover, directly from the definition \eqref{mnasymp}, we have
\begin{align}
 S_\infty(a,0) = 0\quad \mbox{and} \quad S_\infty(-b-2,b)=0,\eqlab{branches}
\end{align}
corresponding to the items \ref{item1} and \ref{item2} in \propref{riccati},
for all $a\in \mathbb R$ and all $b\in \mathbb R$, respectively.
Therefore, at the point \eqref{tr}, two branches of the preimage $S_\infty^{-1}(0)$ intersect transversally in the $(a,b)$-plane, see \figref{riccati}. In the sequel, we
will prove the following claim:

\begin{assertion}\asslab{ass1} {$S_\infty^{-1}(0)\cap U$ consists entirely of the two branches corresponding to the items \ref{item1} and \ref{item2} in \propref{riccati}. Here $U$ is a sufficiently small neighborhood of the point \eqref{tr}.} 
\end{assertion}

The only if part of \propref{riccati} then follows since $\phi\in x^2\mathbb C[[x]]$ is nonanalytic if $S_\infty^0(a,b) \ne 0$. 
In order to prove \assref{ass1}, we use that $S_\infty$ \eqref{Sinfty1} depends analytically on the parameters $(a,b)\in \mathbb R^2$ and subsequently \eqref{branches} to conclude that
\begin{align*}
 S_\infty(a,b) = b(a+b+2)Q(a,b),
\end{align*}
for some analytic function $Q=Q(a,b)$. 
%
The following lemma then proves the \assref{ass1}.
\begin{lemma}
\begin{align*}
Q(-2,0) &=\frac{1}{\Gamma(2)}.
\end{align*}
\end{lemma}
\begin{proof}
Clearly, 
\begin{align*}
 Q(-2,0)=\frac{\partial^2 S_\infty}{\partial b^2}(-2,0).
 \end{align*}
 In the following, we compute the partial derivative by using \eqref{Sinfty1}. For this, we first determine $\frac{\partial \phi_n }{\partial b}(a,b)$ for \eqref{tr}. Upon using that $\phi_n=0$, we find from \eqref{phineqn} that 
 \begin{align*}
  \frac{\partial \phi_n }{\partial b} + \left(n-3\right)\frac{\partial \phi_{n-1} }{\partial b} = \delta_{n2},
 \end{align*}
We conclude that 
\begin{align}\eqlab{phinf20}
\frac{\partial \phi_{n} }{\partial b}(-2,0)=\delta_{n2}.
\end{align} 
In turn, upon using \eqref{Sinfty1}, we obtain that
\begin{align*}
 \frac{\partial^2 S_\infty}{\partial b^2}(-2,0) &= \sum_{n=3}^\infty \frac{(-1)^n \delta_{n4} }{\Gamma(n-2)} = \frac{1}{\Gamma(2)} \ne 0.
\end{align*}
\end{proof}


\begin{remark}\remlab{bif}
 Clearly,
 \begin{align*}
 \frac{\partial S_\infty}{\partial b}(a,0) = \frac{1}{\Gamma(a+2)},
 \end{align*}
 with simple roots at $a\in (-\mathbb N)\setminus\{-1\}$. By the implicit function theorem, branches of $S_\infty^{-1}(0)$ can therefore only bifurcate from $b=0$ at $a\in (-\mathbb N)\setminus\{-1\}$. However, we have not been able to describe the branches analytically. We therefore investigate the situation using numerics in the following section.  
\end{remark}

\begin{figure}[h!]
\begin{center}
\includegraphics[width=.4\textwidth]{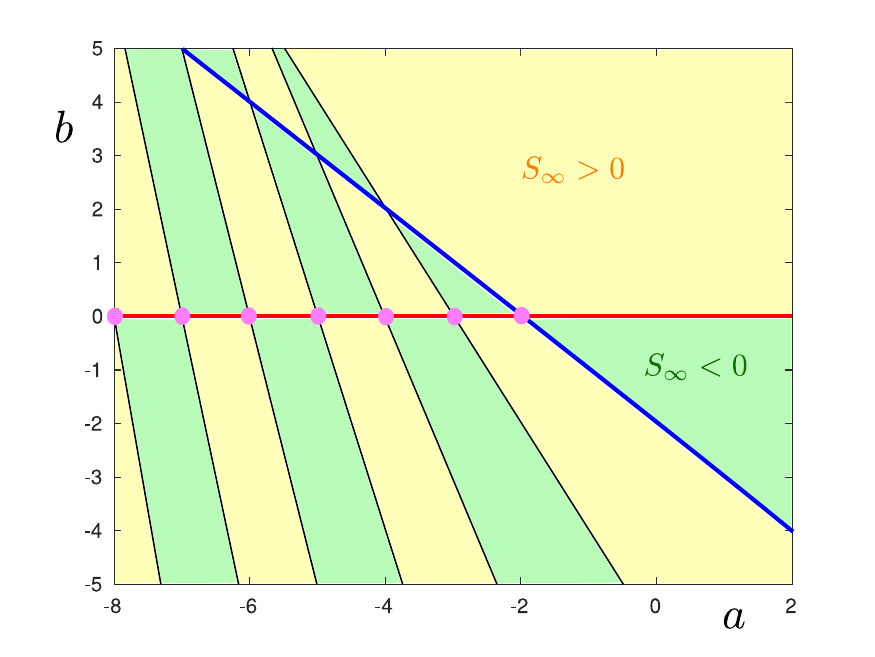}
\end{center}
\caption{\rsp{The sign of $S_\infty(a,b)$ determined by numerical computations (see the main text for details). Here $S_{\infty}>0$ is in yellow whereas $S_{\infty}<0$ is in green. The bifurcations $a\in (-\mathbb N)\setminus \{-1\}$, see \remref{bif}, are shown as purple discs. }}
\figlab{a4_4}
\end{figure}

\subsection{Numerical computations}\seclab{riccati_num}
In \figref{a4_4}, we show the results of numerical computations of $S_\infty$ for the family \eqref{riccati}. We indicate $S_{\infty}>0$ in yellow and $S_{\infty}<0$ in green. To determine $S_\infty$, we have used the approximation $S_\infty\approx S_{N}:=\frac{(-1)^N \phi_N}{\Gamma(N+a)}$ with $N=70$, computing $\phi_n$, $2\le n\le N$, recursively from the recursion relation \eqref{phineqn} for each value of $(a,b)$, and used $4000$ equidistributed grid points in the $(a,b)$-plane. Moreover, the plots were obtained using the contour-plot function in Matlab. In fact, we have used a smoothing of the boundaries where $S_{N}=0$ (black curves) to obtain nicer plots. We did not observe any changes when $N$ was moderately increased, but did notice that around $N=100$ round-off errors begin to influence the plot. We therefore settled on $N=70$.  We have also overlaid the branches in items \ref{item1} and \ref{item2} of \propref{riccati} in red and blue, respectively. We observe the following:
\begin{enumerate}
  \item The result in \propref{riccati} is truly local. There are different bifurcations of $S_\infty(a,b)=0$ emanating from $(a,b)=(-k,0)$ for all $k\in \mathbb N$, $k\ge 3$. These branches (black curves in \figref{a4_4}), that are transverse to $b=0$, are also straight lines in the $(a,b)$-plane but the negative slopes decrease, becoming more vertical as $k$ increasing. Based on the numerical findings, we are confident that these branches are given by 
 \begin{align}
  a=- \frac{1}{k-1}b-k.\eqlab{linek}
 \end{align}
 We did not observe any visible difference when overlaying these curves with \figref{a4_4}. 
 \item Seeing that the lines \eqref{linek} have different slopes, we obtain additional bifurcations for $b>0$ where the lines corresponding to different values of $k$ intersect.
 \item We do not observe further bifurcations for $b\ne 0$ to the right of ``right envelope'' of the branches of \eqref{linek}. We therefore conjecture the following: \textit{The center manifold of \eqref{riccati} is nonanalytic for all $(a,b)\,:\,(a>-\frac{1}{k-1}b-k$ for all $k\in \mathbb N \setminus \{1\})$ and $b\ne 0$}.
\end{enumerate}
I believe that it would be interesting to understand $\phi=\sum_{n=2}^\infty \phi_n x^n\in x^2\mathbb C[[x]]$ along the curves \eqref{linek}. Notice that with $a$ given by \eqref{linek}, then the recursion relation for $\phi_n$ takes the following form:
\begin{align*}
 \phi_n +\left(n-1-\frac{1}{k-1}b-k\right)\phi_{n-1} = b\delta_{n2} +  \sum_{i=2}^{n-2} \phi_i \phi_{n-i}.
\end{align*}
For $k=2$, we know from above that $\phi_n = b^{n-1}$, recall \eqref{phin2}. We expect that it is also possible to solve the recursion for general $k>2$, at least to the level where the convergence of $\phi$ can be determined. However, we will leave further analysis of analytic center manifolds of \eqref{riccati} to future work. 
}

\section{Discussion}\seclab{discuss}

In this paper, we have provided a precise statement about the growth of the coefficients of the formal series expansion of the center manifold of saddle-nodes (in the case of Poincar\'e rank $k=1$). This result, which connects the growth directly with the formal analytic invariant $a$ of the saddle-node, was proven in \cite{MR4445442} for a specific system (while their approach was more general the authors did rely on the specific structure in some places) and in \cite{euler} for a restricted class of systems. Our approach is novel and based on a combination of Borel-Laplace (\rspp{extending} the Banach space framework of \cite{bonckaert2008a}) and the fix-point formulation for the coefficients $\phi_n$, see \eqref{fixpointphin} (which was used by \cite{euler,MR4445442}). In our work, the Borel-Laplace approach does not provide the desired bound $\vert\phi_n\vert \le C\Gamma(n+a)$ directly. Instead we used a subsequent step that bootstraps the bound provided by Borel-Laplace, see \lemmaref{borellaplaceest} and \propref{final}. For experts in applied exponential asymptotics,  Darboux's theorem (in the formulation of Dingle, see \cite{dingle1973a} and \cite[Section 3.1]{crew2024a} for a review) would perhaps appear to be a more direct route to $\vert\phi_n\vert \le C\Gamma(n+a)$. However, for a rigorous construction we believe that this requires a more  detailed analysis of $\Phi$ in $\mathbb C\backslash (-\infty,0]$.  

\rsp{At present, we do not know whether $S_\infty=0$ is sufficient for the center manifold to be analytic in general, recall \textbf{Question I} from the introduction. We have only shown in \propref{riccati} that this was the case locally near \eqref{tr} for the family \eqref{riccati}. I think it would be interesting to understand $\phi=\sum_{n=2}^\infty \phi_n x^n\in x^2 \mathbb C[[x]]$ further along the other curves \eqref{linek} where $S_\infty =0$. Is it convergent? We leave these questions to future work.}

In recent work \cite{new}, the authors generalize \cite[\rsp{Theorem 3}]{bonckaert2008a} to saddle-nodes in $\mathbb C^{n+1}$ with arbitrary Poincar\'e ranks $k\in \mathbb N$. With $n=1$, the equations of \cite{bonckaert2008a} can be reduced to
\begin{align}
 \frac{1}{k}x^{k+1} \frac{dy}{dx} = -y + f(x,y),\quad f(0,0)=f'_y(0,0)=0.\eqlab{orderk}
\end{align}
For this purpose the authors of \cite{new} use a generalization of the Borel-Laplace approach, which enabled a description of $k$-summable normal forms of saddle-nodes in $\mathbb C^{n+1}$. It is well-known that the formal series solutions $y=\sum_{n=1}^\infty \phi_n x^n$ of \eqref{orderk} are Gevrey-$\frac{1}{k}$:
\begin{align*}
 \vert \phi_n\vert \le A T^{n-1} (n-1)!^{\frac{1}{k}}\quad \forall\,n\ge 1,
\end{align*}
see e.g. \cite{bonckaert2008a}.  I am confident that the approach of the present paper can be combined with \cite{new} to provide a more detailed description of the growth of the coefficients $\phi_n$ (in line with \thmref{main} for $k=1$).

\newpage
\bibliography{refs}

\begin{thebibliography}{10}

\bibitem{baldom2013a}
I.~Baldom\'a, O.~Castej\'on, and T.~M. Seara.
\newblock Exponentially small heteroclinic breakdown in the generic hopf-zero
  singularity.
\newblock {\em Journal of Dynamics and Differential Equations}, 25(2):335--392,
  2013.

\bibitem{baldoma2012a}
I.~Baldoma and P.~Martin.
\newblock The inner equation for generalized standard maps.
\newblock {\em {SIAM Journal on Applied Dynamical Systems}}, 11(3):1062--1097,
  2012.

\bibitem{balser1994a}
W.~Balser.
\newblock {\em From Divergent Power Series to Analytic Functions: Theory and
  Application of Multisummable Power Series}.
\newblock Springer, 1994.

\bibitem{bonckaert2008a}
P.~Bonckaert and P.~De~Maesschalck.
\newblock Gevrey normal forms of vector fields with one zero eigenvalue.
\newblock {\em Journal of Mathematical Analysis and Applications},
  344(1):301--321, 2008.

\bibitem{braaksma1992a}
B.~L.~J. Braaksma.
\newblock Multisummability of formal power-series solutions of nonlinear
  meromorphic differential-equations.
\newblock {\em Annales De L Institut Fourier}, 42(3):517--540, 1992.

\bibitem{chapman1998a}
S.~J. Chapman, J.~R. King, and K.~L. Adams.
\newblock Exponential asymptotics and stokes lines in nonlinear ordinary
  differential equations.
\newblock {\em Proceedings of the Royal Society A: Mathematical, Physical and
  Engineering Sciences}, 454(1978):2733--2755, 1998.

\bibitem{chapman2009a}
S.~J. Chapman and G.~Kozyreff.
\newblock Exponential asymptotics of localised patterns and snaking bifurcation
  diagrams.
\newblock {\em Physica D: Nonlinear Phenomena}, 238(3):319--354, 2009.

\bibitem{crew2024a}
S.~Crew and P.~H. Trinh.
\newblock Resurgent aspects of applied exponential asymptotics.
\newblock {\em Studies in Applied Mathematics}, 152(3):974--1025, 2024.

\bibitem{de2020a}
P.~De~Maesschalck and K.~Kenens.
\newblock Gevrey asymptotic properties of slow manifolds.
\newblock {\em Nonlinearity}, 33(1):341--387, 2020.

\bibitem{new}
P.~De~Maesschalck and K.~U. Kristiansen.
\newblock {On $k$-summable normal forms of vector fields with one zero
  eigenvalue}.
\newblock {\em \href{https://arxiv.org/abs/2410.20854}{arXiv:2410.20854}},
  2024.

\bibitem{dingle1973a}
R.~B. Dingle.
\newblock {\em Asymptotic expansions : their derivation and interpretation}.
\newblock Acad. Press, 1973.

\bibitem{dumortier2006a}
F.~Dumortier, J.~Llibre, and J.~C. Artes.
\newblock {\em Qualitative theory of planar differential systems}.
\newblock Springer Berlin Heidelberg, 2006.

\bibitem{ecalle}
J.~Ecalle.
\newblock {\em Introduction aux fonctions analysables et preuve constructive de
  la conjecture de Dulac}.
\newblock Actual. Math., Hermann, Paris, 1992.

\bibitem{gelfreich2001a}
V.~Gelfreich and D.~Sauzin.
\newblock Borel summation and splitting of separatrices for the hénon map.
\newblock {\em Annales De L'institut Fourier}, 51(2):513--567, 2001.

\bibitem{uldall2024a}
K.~U. Kristiansen and P.~Szmolyan.
\newblock {A dynamical systems approach to WKB-methods: The simple turning
  point}.
\newblock {\em Journal of Differential Equations}, 406:202--254, 2024.

\bibitem{euler}
K.~U. Kristiansen and P.~Szmolyan.
\newblock Analytic weak-stable manifolds in unfoldings of saddle-nodes.
\newblock {\em Nonlinearity}, 38(2):025019, 2025.

\bibitem{loray2004a}
F.~Loray.
\newblock Versal deformation of the analytic saddle-node.
\newblock {\em Asterisque}, (297):167--187, 2004.

\bibitem{martinet1983a}
J.~Martinet and J.-P. Ramis.
\newblock Classification analytique des équations différentielles non
  linéaires résonnantes du premier ordre.
\newblock {\em Annales Scientifiques De L'école Normale Supérieure},
  16(4):571--621, 1983.

\bibitem{merle2022b}
F.~Merle, P.~Rapha\"{e}l, I.~Rodnianski, and J.~Szeftel.
\newblock {On blow up for the energy super critical defocusing nonlinear
  Schr{\"o}dinger equations}.
\newblock {\em Inventiones Mathematicae}, 227(1):247--413, 2022.

\bibitem{MR4445442}
F.~Merle, P.~Rapha\"{e}l, I.~Rodnianski, and J.~Szeftel.
\newblock {On the implosion of a compressible fluid {I}: {S}mooth self-similar
  inviscid profiles}.
\newblock {\em Ann. of Math. (2)}, 196(2):567--778, 2022.

\bibitem{merle2022a}
F.~Merle, P.~Rapha\"{e}l, I.~Rodnianski, and J.~Szeftel.
\newblock {On the implosion of a compressible fluid II: Singularity formation}.
\newblock {\em Ann. of Math.}, 196(2):779--889, 2022.

\bibitem{mitschi2016a}
C.~Mitschi and D.~Sauzin.
\newblock {\em Divergent Series, Summability and Resurgence. Monodromy and
  Resurgence: LNM 2153}.
\newblock Springer International Publishing, 2016.

\bibitem{NIST}
F.~Olver, D.~Lozier, R.~Boisvert, and C.~Clark.
\newblock {\em The NIST Handbook of Mathematical Functions}.
\newblock Cambridge University Press, New York, NY, 2010-05-12 00:05:00 2010.

\bibitem{voros1983a}
A.~Voros.
\newblock {The return of the quartic oscillator. The complex WKB method}.
\newblock {\em Annales De L'institut Henri Poincare, Section a (physique
  Theorique)}, 39(3):211--338, 1983.

\bibitem{wasow1965a}
W.~Wasow.
\newblock {\em Asymptotic Expansions for Ordinary Differential Equations}.
\newblock Intercience New York, 1965.

\end{thebibliography}
\bibliographystyle{plain}
\newpage
\appendix 
\section{Basic properties of the gamma function}\applab{gamma}

The gamma function $z\mapsto \Gamma(z)$, defined for $\operatorname{Re}(z)>0$ by
\begin{align}
 \Gamma(z) = \int_0^\infty t^{z-1} e^{-t}dt,\label{eq:Gammadef}
\end{align}
will be important in the following. We therefore collect a few facts (see e.g. \cite[Chapter 5]{NIST} and \cite[\rsp{Lemma 4.3}]{euler}) that will be used throughout the manuscript. 

First, we recall that $\Gamma(n+1)=n!$ for all $n\in \mathbb N_0$, which follows from $\Gamma(1)=1$ and the basic property
\begin{align}
 \Gamma(z+1)=z\Gamma(z) \quad \forall \, \operatorname{Re}(z)>0.\eqlab{Gamma}
\end{align}
The gamma function can be analytically extended to the whole complex plane except zero and the negative integers, i.e. $\Gamma:\mathbb C\backslash (-\mathbb N_0)\rightarrow \mathbb C$. Here $-\mathbb N_0$ are simple poles; specifically
\begin{align}\nonumber
 \lim_{x\rightarrow 0} x\Gamma(x) = \Gamma(1)=1.
\end{align}

In this paper, we will use Stirling's well-known formula:
\begin{align}\eqlab{stirling0}
 \Gamma(x+1) = (1+o(1))\sqrt{2\pi x} \left(\frac{x}{e}\right)^x,
\end{align}
for $x\rightarrow \infty$. The following form
\begin{align}
 \frac{\Gamma(x+b)}{\Gamma(x)} = (1+o(1))  x^b,\eqlab{stirling}
\end{align}
for $b\in \mathbb R$ and $x\rightarrow \infty$,  which can be obtained directly from \eqref{stirling0}, will also be needed. We will also use the Euler integral of the first kind in the form
\begin{align}\eqlab{integral}
 \int_0^{\infty} \frac{s^{y-1}}{(1+s)^{x+y}} ds =  \frac{\Gamma(x)\Gamma(y)}{\Gamma(x+y)}\quad \forall\,x,y>0.
\end{align}
(This is often written in terms of $\int_0^{1}(1-v)^{x-1} v^{y-1} dv$; to obtain \eqref{integral} we use the substitution  $v=s/(s+1)$.)

Next, suppose $b>-2$ and define $$\Gamma_n := \Gamma(n+b)\quad\mbox{for all}\quad n\ge 2.$$ Then the following holds:
 \begin{enumerate}
  \item  There exists a $\overline C=\overline C(b)>0$ such that 
  \begin{align}
   \sum_{k=2}^{n-2}  \Gamma_k \Gamma_{n-k} \le \overline C \Gamma_{n-2} \quad \forall n\,\ge 4.\eqlab{convolutionwkest}
  \end{align}
  \item  Let $\rho>0$. Then there exists a $\overline C=\overline C(b,\rho)>0$ such that 
 \begin{align}
  \sum_{j=2}^{k-2} \rho^{j-k+2} \Gamma_j\le \overline C \Gamma_{k-2}\quad \forall\,k\ge 4.\eqlab{rhowkest}
 \end{align}
  \item  Let $\xi>0$. Then there exists a $\overline C=\overline C(b,\xi)>0$ such that 
\begin{align}
   \sum_{l=2}^{\lfloor \frac{k}{2}\rfloor } \xi^{l-2}  \Gamma_{k-2(l-1)} \le \overline C \Gamma_{k-2} \quad \forall\, k\ge 4.\eqlab{xiwkest}
  \end{align}
 \end{enumerate}
These results follow from Stirling's formula, see also \cite[\rsp{Lemma 4.3}]{euler}.

Finally, the digamma function $\psi$ is defined as the logarithmic derivative of the gamma function:
\begin{align}
 \psi(z): = \frac{\Gamma'(z)}{\Gamma(z)}.\eqlab{digamma}
\end{align}
It will be important to us that $\psi$ is an increasing function of $z>0$:
\begin{align}
 \psi'(z)>0.\label{eq:digammaprop}
\end{align}

\end{document}